\documentclass[a4paper, 11pt]{amsart}

\everymath{\displaystyle}

\usepackage{amsmath,amsfonts,amssymb,amsthm,mathrsfs}
 \usepackage{hyperref}
 \usepackage{bm}
 \usepackage{color}
 \usepackage{esint}
\usepackage{graphicx}
\usepackage{graphics}
\usepackage{geometry}
\usepackage[all]{xy}

\usepackage{tikz}

\usepackage{hyperref}

 \newtheorem{theorem}{Theorem}[section]
 \newtheorem{lemma}[theorem]{Lemma}
 \newtheorem{corollary}[theorem]{Corollary}
 \newtheorem{proposition}[theorem]{Proposition}
 \newtheorem{conjecture}[theorem]{Conjecture}
 \newtheorem{remark}[theorem]{Remark}
 
 \theoremstyle{definition}
 \newtheorem{definition}[theorem]{Definition}
 \newtheorem{question}[theorem]{Question}

\numberwithin{equation}{section}

\newcommand{\p}{\partial}

\newcommand{\Ca}{\mathcal{C}}
\newcommand{\C}{\mathbb{C}}

\newcommand{\cO}{\mathcal{O}}

\newcommand{\cS}{\mathcal{S}}

\newcommand{\R}{\mathbb{R}}

\newcommand{\dS}{\mathbb{S}}

\DeclareMathOperator{\Hilb}{\textbf{Hilb}}

\DeclareMathOperator{\Aut}{Aut}

\DeclareMathOperator{\diam}{diam}

\DeclareMathOperator{\Id}{Id}

\DeclareMathOperator{\Vol}{Vol}

\begin{document}
 \title{Bubbling of K\"ahler-Einstein metrics}

\dedicatory{In Memory of Professor Jean-Pierre Demailly}
\date{\today}
\author{Song Sun}
\address{Department of Mathematics, University of California, Berkeley, CA 94720 } 
\email{sosun@berkeley.edu}

 \maketitle
 \begin{abstract}
 	We prove the finite step termination of bubble trees for singularity formation of polarized K\"ahler-Einstein metrics in the non-collapsing situation.  We also raise several questions and conjectures in connection with algebraic geometry and Riemannian geometry.
 \end{abstract}
 \section{introduction}
The study of K\"ahler-Einstein metrics has been a central topic in geometry for the past few decades. The main existence results on compact K\"ahler manifolds are due to Aubin \cite{Aubin}-Yau \cite{Yau} (in the case $c_1<0$), Yau \cite{Yau} (in the case $c_1=0$) and Chen-Donaldson-Sun \cite{CDS} (in the case $c_1>0$, in terms of the K-stability criterion). One interesting  research direction is lead by the interplay between geometric analysis and complex algebraic geometry of singularities. 
 On the one hand, bridges with algebraic geometry allow one to prove deep results concerning the structure of singularities of K\"ahler-Einstein metrics that are  not yet accessible using traditional techniques in geometric analysis. On the other hand, geometric analytic results  motivate the right  questions concerning singularities in algebraic geometry that would not have been envisioned otherwise and  enable the application of sophisticated algebro-geometric techniques to answer these questions and form algebro-geometric generalizations.

   In this paper we will make some further progress along these lines. More precisely, we aim toward understanding the algebro-geometric meaning of  singularity formation process of K\"ahler-Einstein metrics, building upon previous results and discussion in \cite{DS1, DS2, Donaldson14, SS}. To set the stage we consider a sequence of polarized K\"ahler-Einstein metrics $(X_j, L_j, \omega_j, p_j)$ with $\omega_j\in 2\pi c_1(L_j)$ and $Ric(\omega_j)=c_j\omega_j(c_j\in \{1, 0, -1\})$. When $c_j=0$ we also assume $K_{X_j}$ is holomorphically trivial.  We assume that $\diam(X_j,\omega_j)\leq D$ for some $D\geq0$. The latter in particular implies a uniform \emph{volume non-collapsing} condition, i.e., there exists a $\kappa>0$ such that  \begin{equation}\label{e:vnon}\Vol(B(p_j, 1))\geq \kappa
 \end{equation}
 for all $j$. 
   It is proved in  \cite{DS1} that after passing to a subsequence we obtain a pointed Gromov-Hausdorff limit $(X_\infty, d_\infty, p_\infty)$ which has the structure of a normal projective variety with log terminal singularities and $d_\infty$ is induced by a weak K\"ahler-Einstein metric $\omega_\infty$ on $X_\infty$. It is proved in \cite{DS2} that there is a unique metric tangent cone at $p_\infty$, which we denote by $\mathcal C_0(X_\infty)$. The latter is naturally a normal affine variety endowed with a weak Calabi-Yau (i.e., Ricci-flat K\"ahler) cone metric. Furthermore, the underlying algebraic structure on $\mathcal C_0(X_\infty)$ can be obtained in terms of a 2-step degeneration from the local ring $\mathcal O_{X_\infty, p_\infty}$ defined by the metric.   The cone $\mathcal C_0(X_\infty)$ and an intermediate K-semistable Fano cone $W$ are conjectured in \cite{DS2} to be algebro-geometric invariants of the germ of $X_\infty$ at $p_\infty$, thus leading to a local notion of stability. This conjecture was confirmed by Li-Wang-Xu \cite{LWX}. There are also many  follow-up work on related topics in algebraic geometry; see for example \cite{XZ} and the references therein for more information.   

From the viewpoint of geometric analysis, tangent cones can be regarded as only the \emph{first order} approximation of the singular behavior of solutions to an elliptic non-linear PDE.  In our case the tangent cone $\mathcal C_0(X_\infty)$ only depends on the limit space $X_\infty$, as such it does not encode the dynamic information on how the singularity is formed. To investigate the latter we  consider the limits of the rescaled spaces $(X_j, \lambda_j\omega_j, p_j)$ for \emph{all} possible sequences of integers $\lambda_j\geq1$.  Clearly if $\lambda_j$ stays uniformly bounded then the limit is simply a rescaling of $X_\infty$; if $\lambda_j\rightarrow\infty$, then passing to a subsequence by \cite{DS2}
  we obtain a non-compact pointed Gromov-Hausdorff limit $(Z, d_Z,p_Z)$; the latter is naturally an affine variety with a weak Calabi-Yau metric $\omega_Z$ which is asymptotic to a unique metric cone $\mathcal C_\infty(Z)$ at infinity. Moreover, there is also a unique metric tangent cone $\mathcal C_0(Z)$ at $p_Z$.

From now on we fix the given convergent sequence $(X_j, L_j, \omega_j, p_j)$. For our interest in this article we will always assume $p_\infty$ is a singular point.
 \begin{definition}
 	 A \emph{bubble limit} $(Z, \omega_Z, p_Z)$ associated to the above sequence  is a subsequential pointed Gromov-Hausdorff limit of the sequence $(X_j, \lambda_j\omega_j, p_j)$ for some $\lambda_j\geq 1$. \end{definition}
 	 \begin{remark}
 	 For notational simplicity we will sometimes drop $\omega_Z$ and $p_Z$ when discussing bubble limits, but it should be understood that $Z$ always carries the data $(\omega_Z, p_Z)$ implicitly.
 	 \end{remark}

 	 \begin{remark}\label{r:1-2}
	The notion of (pointed) Gromov-Hausdorff convergence we are referring to in this paper is stronger than the standard definition in terms of metric spaces, since we also include the information on the convergence of complex structures. For example, a priori the same metric space $Z$  may be seen as an affine variety in different ways when passing to different further subsequences. See \cite{DS2}, Section 2.2 for more precise definition of the convergence. 
	\end{remark}

   By definition a bubble limit is endowed with both an algebraic structure and a metric structure. We denote by $\mathfrak R$ the class of all bubble limits modulo the equivalence relation given by holomorphic isometries keeping the base points. Obviously $X_\infty$ is the only compact bubble limit. By a diagonal sequence argument we know $\mathcal C_0(X_\infty)\in\mathfrak R$. Also for any  non-compact $Z\in \mathfrak R$, both the tangent cone $\mathcal C_0(Z)$ and the asymptotic cone $\mathcal C_\infty(Z)$  are in $\mathfrak R$.  Notice that $\mathfrak R$ depends not only  on the limit point $p_\infty$, but also on the sequence of points $p_j$ converging to $p_\infty$.
 \begin{definition}
 	For any  $Z\in \mathfrak R$, its volume density at $p_Z$ is defined to be 
 	$$\Vol_0(Z)\equiv \lim_{r\rightarrow 0}\frac{\Vol(B(p_Z, r))}{\Vol(B_0(r))};$$
 	 when $Z$ is non-compact, its volume density at infinity is defined to be 
 	$$\Vol_\infty(Z)\equiv \lim_{r\rightarrow\infty}\frac{\Vol(B(p_Z, r))}{\Vol(B_0(r))}.$$
 	Here $B_0(r)$ denotes the radius $r$ ball in the Euclidean space $\C^n$.
 \end{definition}
 The Bishop-Gromov inequality implies that these are well-defined positive numbers and we have  $\Vol_0(Z)\geq \Vol_\infty(Z)\geq \Vol_0(X_\infty)$. Moreover, the first equality holds if and only if $Z$ is a metric cone with $p_Z$ as a vertex.  If this occurs  we also denote  $$\Vol(Z)\equiv\Vol_\infty(Z).$$ In particular, $\Vol_0(Z)=\Vol(\mathcal C_0(Z))$ and $\Vol_\infty(Z)=\Vol(\mathcal C_\infty(Z))$. 
 
 \begin{definition}
 	A bubble limit $(\mathcal C, O)$ is a \emph{cone limit} if $\mathcal C$ is a metric cone with $O$ as a vertex. 
 \end{definition}
 \begin{remark}
 As above we will often simply denote by $\mathcal C$ a cone limit. 
 \end{remark}

 Our first result is that
 
 \begin{theorem}[Finiteness of volumes]\label{t:volume finite}
There is a constant $\mathfrak V=\mathfrak V(n,D)$ such that the set
 	$$\mathcal V(\mathfrak R)=\{\Vol_0(Z)|Z\in \mathfrak R\}$$ is a finite set with cardinality bounded by $\mathfrak V$.  \end{theorem}

 \begin{remark}
 It is proved in \cite{MSY, DS2} that $\Vol_0(Z)$ is always an algebraic number, as a consequence of the volume minimization principle discovered by Martelli-Sparks-Yau \cite{MSY}. This was used in \cite{DS2} to show the rigidity of the holomorphic spectrum (and the volume density) in a continuous family of cone limits. The latter plays an indispensable role in the study of tangent cones and asymptotic cones for K\"ahler-Einstein metrics. However the algebraicity of volume does not seem to directly yield Theorem \ref{t:volume finite}. 
 \end{remark}
 \begin{remark}
 It is drawn to the attention of the author that Theorem \ref{t:volume finite} is related to algebro-geometric conjectures on the boundedness and volume of log terminal singularities, see for example \cite{HLL,Zhuang}. The above results give a possible analytic approach to these questions. 
 \end{remark}

 Our proof of Theorem \ref{t:volume finite} uses connections with algebraic geometry. It is natural to ask if the same statement holds in the more general Riemannian setting, see Section \ref{ss:5-2} for more discussion.

 \begin{definition}
 	A \emph{minimal bubble} associated to the above sequence is a bubble limit which is not a cone limit but is asymptotic to $\mathcal C_0(X_\infty)$. 
 \end{definition}

 Next we prove  
\begin{theorem}[Existence of minimal bubbles]\label{t:minimal bubbles}
There is a sequence of scales $\lambda_j\rightarrow\infty$ such that the bubble limits arising from subsequential limits of $(X_j, \lambda_j\omega_j, p_j)$ are minimal bubbles. Furthermore, for any sequence of scales $\lambda_j'\rightarrow\infty$ with $\lambda_j'\lambda_j^{-1}\rightarrow0$, $(X_j, \lambda_j'\omega_j, p_j)$  converges to  $\mathcal C_0(X_\infty)$ in the pointed Gromov-Hausdorff topology.\end{theorem}
\begin{remark}
The second statement and the Bishop-Gromov inequality imply that a minimal bubble has the minimum volume density at infinity among all non-compact bubble limits and it is the unique such minimizer beside the tangent cone $\mathcal C_0(X_\infty)$. This justifies our choice of the terminology  ``minimal bubbles".
\end{remark}
To deduce Theorem \ref{t:minimal bubbles} from Theorem \ref{t:volume finite} the key step is to rule out the possible appearance of a continuous family of cone limits containing $\mathcal C_0(X_\infty)$. Our argument again makes crucial use of algebraic geometry. It can be viewed as an extension of the proof for the uniqueness of tangent cones  in \cite{DS2}. 

\begin{corollary}[Finite step termination of bubble tree] \label{c:bubble tree finiteness}
Given any subsequence of $\{j\}$, passing to a further subsequence the set of bubble limits of $(X_j, \omega_j, p_j)$ modulo metric scalings is a finite ordered set $\{(Z_\alpha, q_{\alpha})\}_{\alpha=1}^{2k}$ with $\Vol_0(Z_{\alpha})$ strictly increasing, 
$(Z_1, q_{1})=(X_\infty, p_\infty)$, $(Z_{2k}, q_{2k})=(\C^n, 0)$, and such that for all $\alpha$, $(Z_{2\alpha}, q_{{2\alpha}})$ is a cone limit, and $\mathcal C_0(Z_{{2\alpha-1}})=Z_{2\alpha}=\mathcal C_\infty(Z_{2\alpha+1})$. Here $k$ is uniformly bounded by an integer depending on $n$ and $D$.
	
\end{corollary}

 It is a natural intriguing question to ask for algebro-geometric interpretations of these minimal bubbles. See Section \ref{ss:5-1} for discussion.

\subsection*{Acknowledgements} The author is grateful to Professor Simon Donaldson and Cristiano Spotti for helpful conversations on the topics of this article some years ago at Stony Brook.  He would like to thank Xuemiao Chen and Junsheng Zhang for relevant discussions in related context.  He also thanks Martin de Borbon, Cristiano Spotti and Junsheng Zhang for valuable comments that improved the exposition. The author is partially supported by NSF Grant DMS-2004261 and the Simons Collaboration Grant in special holonomy.

\section{Preliminaries}
\subsection{Projective and affine algebraic structures}
\label{ss2-1}
Let $(X_j, L_j, \omega_j, p_j)$ be given as in the introduction. Let $\lambda_j\geq1$ be a sequence of integers and consider the rescaled sequence $(X_j, L_j^{\lambda_j},\lambda_j\omega_j, p_j)$. From \cite{DS1, DS2}, after passing to a subsequence we have the following

\

\textbf{(1).} There is a pointed Gromov-Hausdorff limit $(Z, d_Z, p_Z)$. The latter is a complete length space which is equipped with a  regular-singular decomposition $Z=Z^{reg}\cup Z^{sing}$ in the sense of Cheeger-Colding. The regular set $Z^{reg}$ is an open subset which is a smooth complex manifold and the limit metric $d_Z$ is induced by  a smooth K\"ahler-Einstein metric $\omega_Z$ on $Z^{reg}$. The convergence over $Z^{reg}$ is in the smooth Cheeger-Gromov sense. 
	
	\textbf{(2).}  $Z$ is a normal complex-analytic space with structure sheaf given by $\iota_*\cO_{Z^{reg}}$, where $\iota: Z^{reg}\rightarrow Z$ is the natural inclusion map. Moreover,  $\omega_Z$ extends to a global closed positive $(1,1)$ current with locally continuous potentials. The metric singular set $Z^{sing}$ coincides with the complex-analytic singular set of $Z$. 
	
	\textbf{(3).}  $Z$ has only log terminal singularities and $\omega_Z$ satisfies the weak K\"ahler-Einstein equation globally  on $Z$.

	\textbf{(4).}  Fix a sequence of pointed Gromov-Hausdorff approximations realizing the convergence (as in \cite{DS2}, see also the discussion at the end of this subsection). Then there are $0<\lambda_2<\lambda_1\leq1$ (depending only on $n$ and $D$) such that given any holomorphic function $f$ on $B(q, \lambda_1)\subset Z$ and $q_j\in X_j$ converging to $q$, for $j$ large there is a holomorphic function $f_j$ on $B(q_j, \lambda_2)\subset X_j$ that naturally converges to $f$. 

\begin{remark}
Item (4) says that the convergence is complex-analytic in a weak sense. It implies that locally one can define a holomorphic map $F_j: B(q_j, \lambda)\rightarrow \C^N$ for all $j$ large, which is generically one-to-one (so is a normalization map) and converges naturally to a holomorphic embedding $F_\infty: B(q_\infty, \lambda)\rightarrow\C^N$. A non-trivial question is whether $F_j$ can be made an embedding for $ j$ large. Moreover, as remarked in \cite{DS2}, it is not yet known how to describe the convergence of the images of $F_j$ in a certain ``flat" family. For example, are locally the defining equations for the image of $F_j$ given by small perturbations of the defining equations of $Z$? This is a technical point that will be addressed in our special setting in this paper; in general it seems to be a missing link between geometric analysis and complex analytic geometry, which is worth further investigation.
\end{remark}

If $Z$ is compact then it is  simply given by a rescaling of the original limit space $X_\infty$. In this case by \cite{DS1} we know that (possibly passing to a further subsequence) there is a limit hermitian holomorphic line bundle $L_\infty$ over $X_\infty^{reg}$, which extends as an ample $\mathbb Q$-line bundle to the entire $X_\infty$. The latter makes $X_\infty$ a projective variety. Moreover, the convergence $X_j\rightarrow X_\infty$ can be realized algebro-geometrically as the convergence in a fixed Hilbert scheme of projective varieties in some $\mathbb C\mathbb P^M$.  These results have applications to the study of moduli of Fano manifolds, see \cite{SSY, LWX1, Odaka}.

If $Z$ is non-compact, then by \cite{DS2} we have

\begin{itemize}
	\item $Z$ is a normal affine variety whose coordinate ring $R(Z)$ is given by the space of holomorphic functions with polynomial growth at infinity.
	\item For any non-zero $f\in R(Z)$, it has a well-defined degree at infinity
	\begin{equation}\label{e:degree at infinity}\deg(f)\equiv \lim_{r\rightarrow\infty}\frac{\log \sup_{B(p_Z, r)} |f|}{\log r}\geq0.	
	\end{equation}

\end{itemize}

	For later purpose we recall slightly more details on the enhanced (pointed) Gromov-Hausdorff topology we use in this paper, which incorporates the convergence of complex structures (see \cite{DS2} for more details). Given  metric balls $B\subset Z, B'\subset Z'$ for $Z, Z'\in \mathfrak R$,  we define $d_{EGH}(B, B')$ to be the smallest $\epsilon>0$ such that the following holds: there is a metric $d$ on the disjoint union $B\sqcup B'$ extending the metrics on $B$ and $B'$, such that there are open subsets $U\subset B^{reg}, U'\subset B'^{reg}$ which are both $\epsilon$ dense in $B\sqcup B'$ and there exists a diffeomorphism $\chi: U\rightarrow U'$ with $d(x, \chi(x))\leq \epsilon$ and satisfying $$\|\chi^*g'-g\|_{C^2(U)}+\|\chi^*J'-J\|_{C^2(U)}\leq\epsilon,$$
	where $(g, J)$ and $(g', J')$ denote the K\"ahler structure on $U$, $U'$ respectively and $C^2(U)$ is the $C^2$ norm on $U$ defined by the metric $g$. When no such $\epsilon>0$ exists we simply define $d_{EGH}(B, B')=1$. Notice $d_{EGH}$ is not necessarily a metric on $\mathfrak R$ but we do not need this property. It is clear from the definition that $d_{EGH}(B, B')=0$ if and only if there is a holomorphic isometry between $B$ and $B'$. We say a sequence $B_j$ converges to $B_\infty$ if $d_{EGH}(B_j, B_\infty)\rightarrow0$ as $j\rightarrow\infty$. 
	More generally, we say that a sequence $(Z_j, p_{Z_j})\in \mathfrak R$ converges to a limit $(Z_\infty, p_{Z_\infty})$ if for any $R>0$, $d_{EGH}(B(p_{Z_j}, R), B(p_Z, R))\rightarrow 0$ as $j\rightarrow\infty$.

\subsection{Cone limits}\label{ss2-2}
Recall that we denoted by $\mathfrak R$ the class of all bubble limits associated to the sequence $(X_j, L_j, \omega_j, p_j)$, and an element  $(\mathcal C, O)$ in $\mathfrak R$ is a \emph{cone limit} if $\mathcal C$ is a metric cone with $O$ as the vertex.
 We denote by $\mathfrak C$ the subclass of $\mathfrak R$ consisting of all cone limits. Given a metric cone $\mathcal C$ we denote by ${B_{\mathcal C}}$ the unit ball in $\mathcal C$ centered at $O$. Due to the cone structure one sees that the enhanced pointed Gromov-Hausdorff topology on $\mathfrak C$ is induced by the enhanced Gromov-Hausdorff topology on the unit balls $B_{\mathcal C}$.  Namely, a sequence of cone limits $\mathcal C_j$ converge to $\mathcal C$ if and only if  $d_{EGH}(B_{\mathcal C_j}, B_{\mathcal C})\rightarrow 0$. 
 
  Differential geometrically the cone structure is reflected by the dilation invariance. Over the regular set of a cone limit $\mathcal C$ there is a dilation vector field  given by $r\p_r$ (where $r$ is the distance function to $O$). The \emph{Reeb vector field} $\xi\equiv J r\p_r$ is a holomorphic Killing field, which generates a holomorphic isometric action of a compact torus $\mathbb T$ on $\mathcal C$. In \cite{DS2} we discussed the algebro-geometric interpretation of the cone structure. In particular we have

\textbf{(1).} $\mathcal C$ is a  polarized normal affine cone. By this we mean that there is a positive $\R$-grading on the coordinate ring 
	\begin{equation}\label{e:grading}R(\mathcal C)=\bigoplus_{\mu\in \mathcal S(\mathcal C)} R_\mu(\mathcal C),	
	\end{equation}
where $\mathcal S(\mathcal C)\subset \mathbb R_{\geq0}\cap \{\langle\alpha, \xi\rangle|\alpha\in \Gamma^*\}$ for $\Gamma^*\subset Lie(\mathbb T)^*$ the weight lattice of $\mathbb T$.  As in \cite{DS2}, we call $\cS(\Ca)$ the \emph{holomorphic spectrum} of $\mathcal C$. In this paper we will denote 
$$R_I(\mathcal C)=\bigoplus_{\mu\in I}R_\mu(\mathcal C)$$
for any interval $I$.

\textbf{(2).} $f\in R_\mu(\mathcal C)$ if and only if $f$ is homogeneous of degree $\mu$, i.e., $\mathcal L_{\xi}f=\sqrt{-1}\mu f$. For $f\in R(\mathcal C)$, $\deg(f)=\mu$ if and only if the highest degree term in the expansion of $f$ with respect to \eqref{e:grading} has degree $\mu$.

\textbf{(3).}  The $\mathbb R$-grading is induced by a positive multi-grading 
	$$R(\mathcal C)=\bigoplus_{\alpha\in 
	\Gamma^*}\mathcal H_\alpha(\mathcal C),$$
	where  $\mathcal H_{\alpha}(\mathcal C)=R_\mu(\mathcal C)$ for $\mu=\langle\alpha, \xi\rangle$.

 	\textbf{(4).} A choice of a homogeneous generating set of $R(\mathcal C)$ gives rise to an algebraic embedding of $\mathcal C$ as an affine variety in some $\C^N$ and the Reeb vector field $\xi$ extends to a  linear diagonal vector field $\xi=Re(\sqrt{-1}\sum_\alpha\zeta_\alpha z_\alpha\p_{z_\alpha}) $ on $\C^N$ such that each $\zeta_\alpha=\deg(z_\alpha)>0$.

Obviously the decomposition \eqref{e:grading} is orthogonal with respect to the $L^2$ inner product on $R(\mathcal C)$ defined by 
\begin{equation}\label{e:L2}\langle f_1, f_2\rangle\equiv \int_{\mathcal C} f_1\bar f_2e^{-\frac{r^2}{2}} d\mu.	
\end{equation}
 We define the Hilbert function of $\mathcal C$ to be $h_{\mathcal C}: \R_{\geq0}\rightarrow \mathbb Z_{\geq 0}; \mu\mapsto\dim R_\mu(\mathcal C)$.

Now we fix an embedding of $\mathcal C$ in some $\C^N$ as given in item (4) above. Denote by $\mathbb S$ the coordinate ring of $\C^N$. We will make use of the theory of multi-graded Hilbert schemes developed by Haiman-Sturmfels \cite{HS}. There is a projective scheme $\Hilb$ parameterizing affine schemes of the form $\text{Spec}(\dS/\mathcal I)$, where $\mathcal I$ is a $\xi$-homogeneous ideal with Hilbert function $h_{\mathcal I}(d)\equiv\dim (\dS/\mathcal I)_d=h_{\mathcal C}(d)$ for all $d\geq 0$. In particular $\mathcal C$ corresponds to a point $[\mathcal C]\in \Hilb$. By the construction in \cite{HS}, there is a flat universal family over $\Hilb$ and a boundedness result is proved using the theory of Gr\"obner basis. Namely, there exists a $\mu_0=\mu_0(\mathcal C)$ depending only on $\mathcal C$ (more precisely, on the function $h_{\mathcal C}$) such that if a $\xi$-homogeneous ideal $\mathcal I$ satisfies that $h_{\mathcal I}(d)=h(d)$ for $d\leq \mu_0$, then $h_{\mathcal I}(d)=h(d)$ for all $d$. Let $GL_\xi(N)$ be the subgroup of $GL(N;\mathbb C)$ that preserves $\xi$. It naturally acts on $\Hilb$. The stabilizer group of $[\mathcal C]$ agrees with $\Aut(\mathcal C)$, the group of holomorphic transformations of $\mathcal C$ fixing the vertex $O$ and the cone structure.

The above discussion does not make use of the weak Calabi-Yau metric on $\mathcal C$. A crucial consequence of the latter is the volume minimization property of Martelli-Sparks-Yau (see \cite{MSY, DS2}). In our setting we know $\mathcal C$ has $\mathbb Q$-Gorenstein singularities, thus it is a \emph{Fano cone} in the sense of \cite{CoSz}.  This means that there exits a non-vanishing holomorphic section $\Omega$ of $K_{\mathcal C}^l$ (for some positive integer $l$) in a neighborhood $\mathcal U$ of $O$. We may assume $\mathcal U$ is $\mathbb T$ invariant. Now $\mathcal L_{\xi}\Omega=f\Omega$ for some holomorphic function $f$ on $\mathcal U$. As is explained in \cite{DS2} we may perform a weight decomposition $f=\sum_{\alpha}f_\alpha$, where $\mathcal L_\xi f_\alpha=\sqrt{-1}\langle\alpha, \xi\rangle f_\alpha$. Set $$F=\exp(\sum_{\langle\alpha, \xi\rangle\neq 0}-\frac{\sqrt{-1}}{\langle\alpha, \xi\rangle}f_\alpha),$$ then replacing $\Omega$ by $F\Omega$ we may assume $f$ is a constant. Since $\mathbb T$ is compact it follows that $f$ is an imaginary constant.


Given a compact torus $\mathbb T'$ acting on $\mathcal C$ holomorphically such that $\xi\in Lie(\mathbb T')$.  For any $\xi'\in Lie(\mathbb T')$, we may write $\mathcal L_{\xi'}\Omega=f'\Omega$ for a holomorphic function $f'$. Using the fact that $[\xi, \xi']=0$ it follows that $\xi(f')=0$. So $r\p_r(f')=0$. Then $f'$ is globally bounded hence must be a constant. Then we get a linear map $$c: Lie(\mathbb T')\rightarrow \mathbb R$$ such that $\mathcal L_{\xi'}\Omega=\sqrt{-1}c(\xi')\Omega$. 

Now the Calabi-Yau cone metric $\omega_{\mathcal C}$ on $\mathcal C$ satisfies the equation $$\omega_{\mathcal C}^n=e^{U}(\Omega\wedge\overline\Omega)^{1/l},$$
 where $U$  is pluri-harmonic.  Since $\mathcal L_{\xi}\omega_{\mathcal C}=2\omega_{\mathcal C}$, it follows that $$\mathcal L_\xi U=\sqrt{-1}(2n-2c(\xi)l^{-1}).$$ Hence $U$ must be a constant and then 
 \begin{equation} \label{e:normalization}
 	c(\xi)=nl. 
 \end{equation}
  Denote by 
  \begin{equation}\label{e:hyperplane}
  \mathbb H=\{\xi'\in Lie(\mathbb T')|c(\xi')=nl\}, 
  \end{equation}
and by $Lie(\mathbb T')^+$ the set of $\xi'$ such that $\langle\alpha, \xi'\rangle\geq 0$ whenever $\mathcal H_\alpha\neq0$.   Then we have the volume minimization principle  \cite{MSY, DS2}
\begin{itemize}
	\item There is a volume function $$\mathbb V: Lie(\mathbb T')^+\cap \mathbb H\rightarrow \mathbb R_{>0}$$ which 
	is determined by the function $h_{\mathcal C}$ in terms of the index character.  It is a rational function with rational coefficients and $\mathbb V(\xi)$ coincides with the Riemannian volume of the cross section of the cone $\mathcal C$.
	\item $\mathbb V$ is strictly convex. 
	\item $\xi$ is a critical point of $\mathbb V$. 
\end{itemize}
As a consequence we know that $\xi$ and $\mathbb V(\xi)$ are both algebraic.
We will need two results proved in \cite{DS2}:
\begin{itemize}
	\item (Matsushima theorem) For any cone limit $\mathcal C$, the automorphism group $\Aut(\mathcal C)$ is reductive. More precisely, it is given by the complexification of the holomorphic isometry group of $\mathcal C$ fixing the vertex and the cone structure. 
	\item (Bando-Mabuchi theorem) The weak Calabi-Yau metric on a cone limit is uniquely determined by the underlying Fano cone structure, up to holomorphic isometries that preserve the Reeb vector field.
\end{itemize}

\subsection{Algebraic geometry of tangent cones and asymptotic cones} \label{ss2-3}

Given any bubble limit $(Z, \omega_Z, p_Z)\in \mathfrak R$, we consider the rescaled spaces $(Z, \lambda \omega_Z, p_Z)$ for $\lambda\in (0, \infty)$. As $\lambda\rightarrow\infty$, passing to subsequences we obtain bubble limits which are \emph{tangent cones} at $p_Z$.  It is proved in \cite{DS2} that in our setting the tangent cone is unique, which we denote by $(\mathcal C_0(Z), O)$.

More interestingly, in \cite{DS2} it is established a 2-step degeneration theory, i.e.,  $\mathcal C_0(Z)$ can be described algebro-geometrically in terms of a  degeneration from the local germ at $p_Z$ in two steps. There is a valuation on the local ring $\mathcal O_{Z, p_Z}$ defined by the degree function 
$$\deg_{p_Z}(f)\equiv\lim_{r\rightarrow0}\frac{\log\sup_{B(p_Z, r)}|f|}{\log r}\in \mathcal S(\mathcal C_0(Z)).$$
Denote by $R(p_Z)$ the associated graded ring. Then $W=\text{Spec}(R(p_Z))$ defines a Fano cone and there is then a further equivariant degeneration from $W$ to $\mathcal C_0(Z)$.

There is also a dual picture for $\lambda\rightarrow 0$ when $Z$ is non-compact. In this case it follows from \cite{DS2} that $(Z, \lambda\omega_Z,p_Z)$ converges to a unique asymptotic cone $\mathcal C_\infty(Z)$. There is a negative valuation on $R(Z)$ (in the sense of \cite{SZ}) which is related to the degree function by
$$\nu(f)=-\deg(f)=-\lim_{r\rightarrow\infty}\frac{\log\sup_{B(p_Z, r)}|f|}{\log r}\in \mathcal -\mathcal S(\mathcal C_\infty(Z)).$$
Denote by $R_\infty(Z)$ the graded ring associated to the filtration defined by the negative valuation. Then $W_\infty=\text{Spec}(R_\infty(Z))$ defines a Fano cone and there is an equivariant degeneration from $W_\infty$ to $\mathcal C_\infty(Z)$. 

In both cases   $\mathcal C_0(Z)$ and $\mathcal C_\infty(Z)$ can be viewed as K-polystable Fano cones, whereas $W$ and $W_\infty$ can be viewed as intermediate K-semistable Fano cones. For the definitions of stabilities used here one is referred to for example \cite{CoSz, LWX}.
However, there is a sharp contrast between the local and the asymptotic situation. In the local case, it is conjectured in \cite{DS2} that both $W$ and $\mathcal C_0(Z)$ are  invariants of the algebraic singularity of $Z$ at $p_Z$, which should lead to a local notion of stability for general singularities. This conjecture is confirmed in \cite{LX} and \cite{LWX},  based on the earlier work of Li \cite{CLi} which further reformulated and generalized the original conjecture in \cite{DS2} using more algebro-geometric language and the volume minimization principle of Martelli-Sparks-Yau \cite{MSY}. As a result, one can compute both $W$ and $\mathcal C_0(Z)$ in certain explicit examples; in particular, there are examples where $W$ and $\mathcal C_0(Z)$ are genuinely distinct. In the asymptotic setting,  it is the case that neither $W_\infty$ nor $\mathcal C_\infty(Z)$ is an invariant of the affine variety $Z$ only. For example, even $\mathbb C^3$ admits a complete Calabi-Yau metric with non-flat asymptotic cone (see \cite{YLi, CR, Gabor1}).  Moreover, there is a new phenomenon  discovered in \cite{SZ}, which is referred to as ``no semistability at infinity". The latter says that in this case $W_\infty$ must be equal to $\mathcal C_\infty(Z)$ if $\mathcal C_\infty(Z)$ has a smooth cross section; this is conjectured to be always true  in general. 

\section{Discreteness of volumes}
In this section we prove Theorem \ref{t:volume finite}. We start with

\begin{lemma} \label{l:semicontinuity}
	Let $\mathcal C\in \mathfrak C$ be a cone limit and $\mu\in \mathcal S(\mathcal C)$. For any $\epsilon>0$ small, there exists a neighborhood $\mathcal U$ of $\mathcal C$ in $\mathfrak C$ such that for any cone limit $\mathcal C'$ in $\mathcal U$, it holds that $$\dim R_{(\mu-\epsilon, \mu+\epsilon)} (\mathcal C')=\dim R_{\mu}(\mathcal C). $$

	\end{lemma}
\begin{proof}
Suppose otherwise, then we may find $\epsilon>0$ with $[\mu-\epsilon, \mu+\epsilon]\cap \mathcal S(\mathcal C)=\{\mu\}$ and a sequence $(\mathcal C_i, O_i)\in \mathfrak S$ converging to $\mathcal C$ such that 
$\dim R_{(\mu-\epsilon, \mu+\epsilon)} (\mathcal C_i)\neq\dim R_{\mu}(\mathcal C)$ for all $i$.

We first claim this implies that for $i$ large $\dim R_{(\mu-\epsilon, \mu+\epsilon)} (\mathcal C_i)<\dim R_{\mu}(\mathcal C).$ Suppose otherwise, then we can choose a homogeneous orthonormal basis of $R_{ (\mu-\epsilon, \mu+\epsilon)}(\mathcal C_i)$ with respect to the $L^2$ inner product defined by \eqref{e:L2}. Passing to a subsequence we may take limits of the basis elements to obtain an $L^2$ orthonormal set of elements in $R_{\mu}(\mathcal C)$ whose cardinality is strictly bigger than $\dim R_{[\mu-\epsilon, \mu+\epsilon]}(\mathcal C)$. This is a contradiction.

	Given a homogeneous holomorphic function $f\in R_\mu(\mathcal C)$, by the discussion in Section \ref{ss2-1} for $i$ large we may find a holomorphic function $f_i$ on the unit ball $B_{\mathcal C_i}$ that naturally converges to $f$. Now using the cone structure on $\mathcal C_i$ we may write $f_i=\sum_{\alpha} f_{i, \alpha}$, where each $f_{i, \alpha}$ is homogeneous with respect to $\xi_i$ with weight $\alpha$. The decomposition is $L^2$ orthogonal so in particular each term $f_{i, \alpha}$ has uniformly bounded $L^2$ norm on $B_i$. We may write $f_{i}=u_{i}+v_{i }$, where $u_{i}$ consists of homogeneous terms of degree in $[\mu-\frac{\epsilon}{2}, \mu+\frac{\epsilon}{2}]$, and $v_{i}$ consists of homogeneous terms of degree not in $[\mu-\frac{\epsilon}{2}, \mu+\frac{\epsilon}{2}]$. It follows that $\|v_i\|_{L^2(B(O_i, \frac{3}{4}))}$ converges to 0  as $i\rightarrow\infty$. So we know $u_i$ converges to $f$ uniformly in the half size ball. 	Applying this construction to a basis of $R_\mu(\mathcal C)$ would imply that for $i$ large $\dim R_{(\mu-\epsilon, \mu+\epsilon)} (\mathcal C_i)\geq\dim R_{\mu}(\mathcal C).$ Contradiction. 
	\end{proof}

Now we suppose a sequence of cone limits $\mathcal C_i\in \mathfrak C$ converge to a cone limit $\mathcal C$. Our next goal is to realize this convergence algebraically.  
First of all, we choose $\underline\mu\notin \mathcal S(\mathcal C)$ large so that $R(\mathcal C)$ is generated by $R_{[0, \underline\mu)}(\mathcal C)$. 
By Lemma \ref{l:semicontinuity} we may choose $\epsilon>0$ small such that $\dim R_{[0,\underline\mu+\epsilon)}(\mathcal C_i)=\dim R_{[0, \underline\mu)}(\mathcal C)$  for $i$ large. Now we choose an $L^2$ orthonormal basis of  $R_{[0, \underline\mu+\epsilon)}(\mathcal C_i)$. Passing to a subsequence we may assume that they converge to an $L^2$ orthonormal basis of $R_{[0, \underline\mu)}(\mathcal C)$. Using these functions we obtain holomorphic maps $F_i: \mathcal C_i\rightarrow \C^N$ which converge naturally to $F:\mathcal C\rightarrow\C^N$ under the Gromov-Hausdorff convergence. By the choice of $\underline\mu$ we know $F$ is an algebraic embedding onto a normal affine variety $W\subset \C^N$.
It follows similar to \cite{DS2} that  for $i$ large $F_i$ is generically one-to-one onto a reduced affine variety $W_i\subset \C^N$. Hence $F_i$ is the normalization map. In particular, it is an isomorphism if we know $W_i$ is normal. Clearly $W_i$ converges in the (local) Hausdorff sense to $W$.

By construction $W_i$ is invariant under a linear vector field $\xi_i$ on $\C^N$ which converges to $\xi$ as $i\rightarrow\infty$. It is easy to see that by suitable linear transformations on $\C^N$ (i.e., by choosing a basis of $R_{[0, \underline\mu+\epsilon)}(\mathcal C_i)$ consisting of $\xi_i$-homogeneous functions) we may assume $\xi$ and each $\xi_i$  are in fact diagonalized with respect to the standard coordinates on $\C^N$.

Fix a Euclidean metric $\omega_0$ on $\C^N$ so that it induces a weighted $L^2$ inner product on the space of all polynomial functions: 
\begin{equation}\label{e:L22}\langle F_1, F_2\rangle =\int_{\C^N} F_1\overline F_2e^{-\frac{|z|^2}{2}}\omega_0^n.	
\end{equation}

Now we set some notations. We say  a holomorphic polynomial function $f\in \mathbb S$ is $\xi_i$ (reps. $\xi$)-homogeneous if $\mathcal L_{\xi_i}f=\sqrt{-1} \mu f$ (resp.$\mathcal L_{\xi}f=\sqrt{-1} \mu f$)  for some $\mu\geq 0$; in this case  we denote $\deg_{\xi_i}(f)=\mu$ (resp. $\deg_\xi(f)=\mu$). Clearly a monomial is both $\xi$-homogeneous and $\xi_i$-homogeneous. Given a polynomial $f$, its weighted degree with respect to $\xi_i$ (resp. $\xi$) is defined in the obvious way using the weighted degree on monomials in terms of $\deg_{\xi_i}$ (resp. $\deg_\xi$). 
We denote by $\mathbb S_{i,\mu}$ (resp. $\mathbb S_\mu$) the space of polynomial functions on $\C^N$ whose weighted degree with respect to $\xi_i$ (resp. $\xi$) is at most $\mu$, by $\mathcal I_i$ (resp. $\mathcal I$) the ideal of $\mathbb S$ consisting of polynomial functions on $\C^N$ which vanish on $W_i$ (resp. $W$) and by $\mathcal I_{i,\mu}$ (resp. $\mathcal I_\mu$) the  finite dimensional subspace of $\mathcal I_i$ (resp. $\mathcal I$) consisting of those polynomials $f$ with weighted degree at most $\mu$ with respect to $\xi_i$ (resp. $\xi$).

Let $\mu_0=\mu_0(\mathcal C)$ be given in Section \ref{ss2-2}.  Choose $\mu_1>\max(\underline\mu,\mu_0)$ sufficiently large and $\epsilon>0$ small so that $\mathcal I_{\mu_1}$ generates $\mathcal I$, $\mathcal I_{\mu_1}=\mathcal I_{\mu_1+\epsilon}$ and $\dim \mathbb S_{i,\mu_1+\epsilon}=\dim \mathbb S_{\mu_1+\epsilon}=\dim \mathbb S_{\mu_1}$.

\begin{lemma}\label{l3-2}For $i$ large we have $\dim \mathcal I_{i,\mu_1+\epsilon}=\dim \mathcal I_{\mu_1+\epsilon}$.	
\end{lemma}
\begin{proof} To see this,  on the one hand, for each $i$ we can choose an $L^2$ orthonormal basis of $\mathcal I_{i,\mu_1+\epsilon}$ with respect to the inner product defined in \eqref{e:L22}. By passing to a subsequence we can take limits of the basis elements, which then implies that $\dim \mathcal I_{ \mu_1+\epsilon}\geq \dim \mathcal I_{i, \mu_1+\epsilon}$ for $i$ large. On the other hand, for each $i$ we choose an $L^2$ orthonormal basis of $R_{[0,\mu_1+\epsilon)}(\mathcal C_i)$. Then by passing to a subsequence we can take limits under the Gromov-Hausdorff convergence and obtain an $L^2$ orthonormal set of holomorphic functions on $\mathcal C$. This implies that for $i$ large $\dim R_{[0, \mu_1+\epsilon)}(\mathcal C)\geq \dim R_{[0,  \mu_1+\epsilon)}(\mathcal C_i)$. But we have $$\dim R_{[0, \mu_1+\epsilon)}(\mathcal C_i)\geq \dim \mathbb S_{i,\mu_1+\epsilon}-\dim \mathcal I_{i,\mu_1+\epsilon},$$ and  $$\dim R_{[0, \mu_1+\epsilon)}(\mathcal C)=\dim \mathbb  S_{\mu_1+\epsilon}-\dim \mathcal I_{ \mu_1+\epsilon}.$$ So we see the desired identity must hold. 
	
\end{proof}

As a consequence we can choose a basis $\{f_{i,a}\}$ of $\mathcal I_{i, \mu_1+\epsilon}$ such that  each $f_{i, a}$ is $\xi_i$-homogeneous and as $i\rightarrow\infty$, they converge to a basis $\{f_{ a}\}$ of $\mathcal I_{ \mu_1}$ such that each $f_{ a}$ is $\xi$-homogeneous. By the choice of $\mu_1$  we know the latter generates the ideal $\mathcal I$.

Denote by $\mathcal I_i'$ the ideal of $\mathbb S$ generated by $\mathcal I_{i, \mu_1+\epsilon}$. Let $W_i'$ be the affine scheme in $\C^N$ defined by $\mathcal I_i'$.  Since $W_i$ is $\xi_i$ invariant we see that $\mathcal I_{i+\epsilon}$ is spanned by $\xi_i$-homogeneous elements, so $\mathcal I_i'$ is $\xi_i$-homogeneous. In particular $W_i'$ is also $\xi_i$ invariant. A priori $W_i'$ could be strictly bigger than $W_i$ as a scheme, i.e., $\mathcal I_i'$ might be strictly contained in $\mathcal I_i$.

\begin{lemma}\label{l3-3}
For $i$ large $W_i'$ is  $\xi$ invariant and $W$ is $\xi_i$ invariant. 
	\end{lemma}

\begin{proof}
Since each $f_{i, a}$ is $\xi_i$-homogeneous and $f_{i, a}$ converges to $f_{a}$ as $i\rightarrow\infty$,  $f_{ a}$ must also be $\xi_i$-homogeneous for $i$ large.  Thus $W$ is $\xi_i$ invariant.

To show $W_i$ is $\xi$ invariant for $i$ is large we argue by contradiction. Suppose not, without loss of generality we may assume that for some fixed $a$ and for all $i$ large, $f_{i, a}$ is not $\xi$-homogeneous.  Passing to a subsequence we may further assume that there is a fixed monomial $h\in \mathbb S_{\mu_1}$  which appears in the monomial expansion of $f_{i,a}$ for all $i$ and such that $\deg_{\xi}(h)\neq \deg_{\xi}(f_{a})$. Since $\xi_i\rightarrow\xi$, we see that for $i$ large, $\deg_{\xi_i}(h)\neq \deg_{\xi_i}(f_{ a})$. This contradicts the fact that $f_{i,a}$ is $\xi_i$-homogeneous. 
\end{proof}

Denote by $\mathbb T$ the compact torus in $GL(N;\mathbb C)$ generated by $\xi$. Denote by $\Hilb$ the multi-graded Hilbert scheme defined by the Hilbert function $h_{\mathcal C}$ with respect to the $\mathbb T$ action. 
Lemma \ref{l3-2} implies that for $i$ large, $h_{\mathcal I_i'}(d)=h_{\mathcal C}(d)$ for all $d\leq \mu_1$. Then by the result of Haiman-Sturmfels \cite{HS} (as reviewed in Section \ref{ss2-2}) we see $W_i'$ is also represented by a point $[W_i']$ in $\Hilb$. In particular, $h_{\mathcal I'}(d)=h_{\mathcal C}(d)$ for all $d\geq 0$. Furthermore, $[W_i']\rightarrow [\mathcal C]$ as $i\rightarrow\infty$.  Since the universal family over $\Hilb$ is flat and normality is an open condition  it follows that for $i$ large $W_i'$ is normal. In particular it is reduced and irreducible, but we have $\dim W_i=\dim W_i'$, so we conclude that $W_i'=W_i$. It follows that $F_i$ must be an isomorphism.

\begin{proposition}\label{p:volume locally constant}
	For $i$ large we have $\xi_i=\xi$ and $\Vol(\mathcal C_i)=\Vol(\mathcal C)$.
\end{proposition}
\begin{proof}
 It suffices to show that given any subsequence of $\{i\}$, by passing to a further subsequence for $i$ large it holds that $\xi_{i}=\xi$ and $\Vol(\mathcal C_i)=\Vol(\mathcal C)$. Without loss of generality we may assume the subsequence is the original sequence. First by the openness of rational singularities \cite{Elk} we may assume $W_i$ has rational singularities for $i\geq i_0$. 
We take the set of points $P=\{[W_i]|i\geq i_0\}\in \Hilb$ and take its Zariski closure $\mathcal   Z$. Denote by $\pi:\mathcal W\rightarrow \mathcal Z$ the restriction of the flat universal family. Since each affine cone by definition contains the origin, we have an obvious section $\sigma$ of $\pi$.  By \cite{FEM} we can find an integer $l>0$ and a Zariski open $\mathcal O\subset \mathcal Z$ such that for any $t\in \mathcal O$, the corresponding affine cone $\mathcal W_t=\pi^{-1}(t)$ is normal with rational singularities and $K_{\mathcal W_t}^{l}$ is trivial. Moreover, we may take  a holomorphic section $\Omega$ of $K_{\mathcal W/\mathcal O}^{l}$ over a neighborhood of $\sigma$. Using the $\mathbb T$ action we can as in Section \ref{ss2-2} modify $\Omega_t$ simultaneously for all $t\in \mathcal O$  so that  $\mathcal L_{\xi}\Omega_t=\sqrt{-1}c_t\Omega_t$, where $c_t$ is a real-valued function which depends holomorphically on $t$. It follows that $c_t$ must be a constant in $t$. So we have $c_t=c_0=nl$. 

This implies that passing to a subsequence for $i$ large we know both $\xi$ and $\xi_i$ belong to the same hyperplane $ \mathbb H=\mathbb H_i\subset Lie(\mathbb T)$ (as defined in \eqref{e:hyperplane}). On the other hand  by the  properties of the volume function listed in Section \ref{ss2-2} we see  $\mathbb V_i=\mathbb V$. Since $\mathbb V$ is strictly convex and both $\xi$ and $\xi_i$ are critical points of $\mathbb V|_{\mathbb H}$, we conclude  that $\xi=\xi_i$. Consequently, we also have $\Vol(\mathcal C_i)=\Vol(\mathcal C)$.
\end{proof}

By a simple contradiction argument the above discussion also gives
\begin{corollary}\label{c:locally constant}
	There exists $\underline\epsilon>0$ depending only on $n$ and $D$ such that for a given cone limit $\mathcal C\in \mathfrak C$, all the cone limits $\mathcal C'\in \mathfrak C$ with $d_{EGH}(B_{\mathcal C}, B_{\mathcal C'})<\underline\epsilon$ can be parametrized by  a connected component in a multi-graded Hilbert scheme $\Hilb$ of $\C^N$ such that the Reeb vector field of $\mathcal C'$ is given by the restriction of a common diagonal linear vector field $\xi$ on $\C^N$.  
\end{corollary}
\begin{proof}[Proof of Theorem \ref{t:volume finite}]
Suppose otherwise, then we can find a sequence $Z_i\in \mathfrak R$ such that $\Vol_\infty(Z_i)$ are all distinct. Let $\mathcal C_i$ be the unique asymptotic cone of $Z_i$. Then $\Vol(\mathcal C_i)=\Vol_\infty(Z_i)$. Passing to a subsequence we may assume $\mathcal C_i$ converges to a limit $\mathcal C$. Then we get a contradiction with Proposition \ref{p:volume locally constant}. 

From the argument we can also see that the bound only depends on $D$, which guarantees the non-collapsing property and  Gromov-Hausdorff compactness of the relevant cones.  
	
\end{proof}

\begin{remark}
The above arguments work as long as \cite{DS2} can be applied so may be useful in more general setting. For example, Theorem \ref{t:volume finite}  holds for general Ricci-flat K\"ahler cones with isolated singularities or with orbifold singularities along rays, without assuming that they arise as bubble limits. We leave this direction for interested readers to explore further. 
\end{remark}

\section{Existence of minimal bubbles}

We  adopt the notation in the introduction. Let $(X_j, L_j, \omega_j, p_j)$ be a sequence of polarized K\"ahler-Einstein manifolds converging to $(X_\infty, \omega_\infty, p_\infty)$.   By \cite{DS1} the convergence can be understood algebro-geometrically in the flat universal family over a fixed Hilbert scheme of $\mathbb P^M$. As a consequence there is a neighborhood $U$ of $p_\infty$ in $\mathbb P^M$ such that $X_\infty\cap U$ is cut out by finitely many holomorphic functions $\{h^a\}_{a=1}^m$ on $U$ and for $j$ large $X_j\cap U$ is cut out by $\{h_j^a\}_{a=1}^m$. Moreover, for each $a$, $h_j^a$ converges uniformly to $h^a$ over $U$ as $j\rightarrow\infty$.

Let $\mathcal C_0(X_\infty)$ be the unique tangent cone at $p_\infty$, and denote $V\equiv\Vol_0(X_\infty)$, $\mathcal S\equiv \mathcal S(\mathcal C_0(X_\infty))$. Let $\mu_1=\mu_1(\mathcal C_0(X_\infty))$ be given in Lemma \ref{l3-2} and fix $\epsilon>0$ such that $$[V-\epsilon, V+\epsilon]\cap \mathcal V(\mathfrak R)=\{V\}.$$

Choose $\mu_2\geq\mu_1$ such that the graded ring $R(p_\infty)$ is generated by functions in degrees at most $\mu_2$. For $\mu_*\geq \mu_2$, we choose a finite set $\{f^\alpha\}_{\alpha=0}^N\subset \mathcal O(p_\infty)$ which descends to a basis of $R_{[0, \mu_*]}(p_\infty)$ and such that $\deg(f^\alpha)\leq \deg(f^\beta)$ whenever $\alpha\leq\beta$. Notice $f^0$ is a constant.  Choose an extension of $f^\alpha$ to a holomorphic function $\hat f^\alpha$ on a neighborhood of $p_\infty$ in $\mathbb P^M$. We may fix $\mu_*$ large so that $\{\hat f^\alpha\}_{\alpha\geq 1}$ define a holomorphic embedding of a neighborhood $U\subset \mathbb P^M$.  For $j$ large we denote by $f^\alpha_j$ the restriction of $\hat f^\alpha$ to $U\cap X_j$. Clearly, as $j\rightarrow\infty$, $f^\alpha_j$ converges naturally to $f^\alpha$ under the convergence $X_j\rightarrow X_\infty$. 

 Denote $$\delta=\inf\{|\mu-\nu||\mu, \nu\in \mathcal S\cap [0, \mu_*], \mu\neq\nu\}$$ 
 For $\lambda>0$ denote by $X_{j, \lambda}$ the rescaled space $(X_j, p_j, \lambda^2\omega_j)$ and by $B_{j, \lambda}$ the unit ball in $X_{j, \lambda}$ centered at $p_j$. We denote $$\underline{\Vol}(B_{j, \lambda})=\frac{\Vol(B_{j, \lambda})}{\Vol(B_0(1))},$$
 where $B_0(r)$ denotes the radius $r$ ball in the Euclidean space $\C^n$.

From the Bishop-Gromov monotonicity and Colding's volume convergence theorem \cite{Colding} we see that there exists a $\underline\lambda>0$ such that $\underline{\Vol}(B_{j, \underline\lambda})\geq V-\epsilon/4$ for all $j$ large. Denote by $\overline\lambda_j$ the largest $\lambda$ such that $\underline{\Vol}(B_{j,\lambda}) -V\leq\epsilon/4$.  Notice $\overline\lambda_j\rightarrow\infty$ as $j\rightarrow\infty$.  We may assume $\underline\lambda$ is large so that $B_{j, \underline\lambda}\subset U$ for all $j$ large.

\begin{lemma}\label{lem1.9}
	Given a sequence $\lambda_j\in[\underline\lambda, \overline\lambda_j]$ with $\lambda_j\rightarrow\infty$ and $\lambda_j^{-1}\overline\lambda_j\rightarrow\infty$, any subsequential limit of $X_{j, \lambda_j}$ is a cone limit. \end{lemma}
\begin{proof}
Passing to a subsequence we may take a pointed Gromov-Hausdorff limit $Z$. The assumptions imply that the asymptotic cone $\mathcal  C_\infty(Z)$ and the tangent cone $\mathcal C_0(Z)$ at $p_Z$ both have volume $\epsilon$-close to $V$. It follows from our choice of $\epsilon$ that $\Vol(\mathcal  C_\infty(Z))=\Vol(\mathcal C_0(Z))=V$. Hence $Z$ itself is a cone limit. 
\end{proof}

 Denote 
by $\mathfrak C'$ the class of all cone limits arising from the above lemma. Notice we have $\mathcal C_0(X_\infty)\in \mathfrak C'$.  It is not difficult to see that $\mathfrak C'$ is connected.

\begin{lemma}
All the cone limits in $\mathfrak C'$ can be parametrized by a fixed multi-graded Hilbert scheme $\Hilb$ of $\C^N$ such that the Reeb vector fields are given by the restriction of a common diagonal linear vector field $\xi$ on $\C^N$.	 In particular, all the cone limits in $\mathfrak C'$ have the same holomorphic spectrum  which coincides with $\mathcal S$ and have volume density $V$. 
\end{lemma}

\begin{proof}
We suppose otherwise, then by Lemma \ref{lem1.9} after passing to a subsequence we may find scales $\underline\lambda<\lambda_j'<\lambda_j''<\overline\lambda_j$ such that $X_{j, \lambda_j'}\rightarrow \mathcal C'$ and $X_{j, \lambda_j''}\rightarrow\mathcal C''$ but $\mathcal C'$ and $\mathcal C''$ do not belong to the same multi-graded Hilbert scheme.  Now by a simple contradiction argument we can see that for all $j$ large and  all $\lambda\in [\lambda_j', \lambda_j'']$  there is a $\mathcal C_{j, \lambda}\in \mathfrak C'$ with  $d_{EGH}(B_{j, \lambda},B_{\mathcal C_{j, \lambda}})\leq \underline\epsilon/10$ such that $C_{j, \lambda_j'}=\mathcal C'$ and $C_{j, \lambda_j''}=\mathcal C''$, where $\underline\epsilon$ is given in Corollary \ref{c:locally constant}. Then it follows easily that for all  $\lambda\in [\lambda_j', \lambda_j'']$ we can find $\delta_{\lambda}>0$ such that $d_{EGH}(B_{\mathcal C_{j, \beta}}, B_{\mathcal C_{j, \lambda}})<\underline\epsilon$ for all $\beta\in [\lambda-\delta_{ \lambda}, \lambda+\delta_{ \lambda}]$,  So by Corollary \ref{c:locally constant} we see that for all $\beta\in[\lambda-\delta_{ \lambda}, \lambda+\delta_{ \lambda}]$, $\mathcal C_{j, \beta}$ is parametrized in a fixed multi-graded Hilbert scheme of $\C^N$ with the same Reeb vector field. Since $[\lambda_j', \lambda_j'']$ can be covered by finitely many such intervals, we obtain a contradiction. 
\end{proof}

The following is the main result of this section.
\begin{theorem}\label{t:isolation}
 $\mathcal C_0(X_\infty)$ is an isolated point in $\mathfrak C'$. 
\end{theorem}

\begin{proof}[Proof of Theorem \ref{t:minimal bubbles}]
	We first prove that the subsequential limits of $X_{j, \overline\lambda_j}$ are minimal bubbles. Let $Z$ be such a subsequential limit. By our choice of $\overline \lambda_j$ and Bishop-Gromov monotonicity we know $\Vol_\infty(Z)<\Vol_0(Z)$, so $Z$ is not a cone limit, so it suffices to show  $\mathcal C_\infty(Z)= \mathcal C_0(X_\infty)$. Suppose otherwise, notice $\mathcal C_0(X_\infty)$ is the bubble limit associated to a sequence of scales $\lambda_j\in (\underline\lambda, \overline\lambda_j)$, so for each integer $k\gg0$ and $j$ large we may find $\lambda_{j,k}\in [\underline\lambda, \overline\lambda_j]$ such that $d_{EGH}(B_{j, \lambda_{j,k}}, { B}_{\mathcal C_0(X_\infty)})=k^{-1}$. Taking a subsequential limit we obtain a cone limit $\mathcal C_k$ with $d_{EGH}( { B}_{\mathcal C_{k}}, {B}_{\mathcal C_0(X_\infty)})=k^{-1}$. This contradicts Theorem \ref{t:isolation}.
	
	 The second statement can be proved similarly.
\end{proof}

The rest of this section is devoted to the proof of Theorem \ref{t:isolation}. 
 First we need to slightly extend the convexity result in \cite{DS2}. Define $$\underline k\equiv\log_2(\sigma\underline\lambda), \ \ \ \ \overline k_j\equiv\log_2(\sigma^{-1}\overline\lambda_j).$$ The proof of the lemma below follows from exactly the same contradiction argument as in \cite{DS2}, Proposition 3.7. So we omit the proof.  
\begin{lemma}\label{lem1.10}
	Given any fixed $\mu\in (0, \infty)\setminus \mathcal S$, there exist  a $\underline j\geq 0$ and $\sigma>0$ such that for all $j\geq \underline j$ and $k\in [\underline k, \overline k_j]$, if a holomorphic function $f$ defined on $B_{j, 2^k}$ satisfies $\|f\|_{B_{j, 2^{k+1}}}\geq2^{-\mu}\|f\|_{B_{j, 2^k}}$, then  $\|f\|_{B_{j, 2^{k+2}}}> 2^{-\mu}\|f\|_{B_{j, 2^{k+1}}}$, where $\|\cdot\|$ denotes the natural $L^2$ norm.
\end{lemma}
From the uniform convergence of $f^\alpha_j$ to $f^\alpha$ that  we can choose $\underline k>0$ such that
  $$\|f^\alpha_j\|_{B_{j, 2^{\underline k+1}}}\geq2^{-\deg(f^\alpha)-\frac{\delta}{2}}\|f^\alpha_j\|_{B_{j,2^{\underline k}}}$$ for all $\alpha$ and for all $j$ large. It follows from the lemma that 
  $$\|f^\alpha_j\|_{B_{j, 2^{k+1}}}\geq2^{-\deg(f^\alpha)-\frac{\delta}{2}}\|f^\alpha_j\|_{B_{j,2^{k}}}$$ for all $k\in [\underline k, \overline k_j]$ when $j$ is large. 
  On $B_{j, k}$ we define $$\widetilde f^\alpha_{j,k}=\frac{f_j^\alpha}{\|f_j^\alpha\|_{B_{j, 2^k}}}.$$
Using Lemma \ref{lem1.10} and the arguments in \cite{DS2} we see that when $j,k\rightarrow\infty$,  passing to a subsequence each $\widetilde f^\alpha_{j,k}$ converges to a non-trivial limit holomorphic function $f^\alpha_\infty$ on some limit $Z\in \mathfrak R$. Moreover, we have $\deg(f_\infty^\alpha)\leq \deg(f^\alpha)$.

Now for a holomorphic function on $B_{j, 2^k}$ we denote by $\Pi_{j, k}(f)$  the $L^2$ orthogonal projection of $f$ to the orthogonal complement of $f^1$, then normalized so that $\|\Pi_{j, k}f\|_{B_{j, 2^k}}=1$. The following Lemma can be proved similarly to \cite{DS2}, Proposition 3.11. Again we omit the proof here. 

\begin{lemma} \label{l:convexity2}
Given any  $\mu\in (0, \infty)\setminus \mathcal S$, by possibly enlarging $\underline j$ and $\sigma$, we have that  for all $j\geq \underline j$ and $k\in [\underline k, \overline k_j]$, if $f$ is a holomorphic function $f$ on $B_{j, 2^k}$ such that  	$\|\Pi_{j, k+1}f\|_{B_{j, 2^{k+1}}}\geq2^{-\mu}\|\Pi_{j, k}f\|_{B_{j,2^k}}$, then  $\|\Pi_{j, k+2}f\|_{B_{j, 2^{k+2}}}> 2^{-\mu}\|\Pi_{j,k+1}f\|_{B_{j, 2^{k+1}}}$. 
\end{lemma}

Given this Lemma,  as in \cite{DS2, CS} arguing by induction and by possibly enlarging $\underline j$ and $\underline k$ again,  we may find functions $\{\widetilde f^\alpha_{j, {k}}\}$ ($0\leq\alpha\leq N$, $j\geq \underline j$ and $k\in [\underline k, \overline k_j]$)  such that 
\begin{itemize}
\item $\widetilde f^\alpha_{j,k}=const$ when $\alpha=0$;
	\item $\widetilde f^\alpha_{j, {k}}$ belongs to the linear span of $f^\beta_j$ for $\beta\leq\alpha$;
	\item $\{\widetilde f^\alpha_{j, {k}}\}$ is $L^2$ orthonormal over over $B_{j, 2^k}$;
	\item as $j,k\rightarrow\infty$, passing to a subsequence $\widetilde f^\alpha_{j, {k}}$ converges uniformly to a limit $f_\infty^\alpha$ on a limit $Z\in \mathfrak R$ with $\deg(f_\infty^\alpha)\leq\deg(f^\alpha)$. 
\end{itemize}
 In particular, $\{f_\infty^\alpha\}$ are $L^2$ orthonormal on the unit ball $B(p_Z, 1)$ in $Z$. We denote by $R_{\leq \mu}$ the space of holomorphic function on $Z$ with degree at most $\mu.$

\begin{lemma}
	For all $\alpha$ we have $\deg(f^\alpha_\infty)=\deg(f^\alpha)$ and $R_{\leq\mu_*}(Z)$  is precisely spanned by $\{f^\alpha_\infty\}$.  Moreover, if $Z$ is a cone limit, then $\{f^\alpha_\infty\}$ is a homogeneous $L^2$ orthonormal basis.
\end{lemma}
\begin{proof}
Since $\deg(f_\infty^\alpha)\leq\deg(f^\alpha)$ for all $\alpha$, it follows that $\dim R_{\leq\mu}(Z)\geq \dim R_{[0,\mu]}({p_\infty})$ for all $\mu\leq \mu_*$. But since the asymptotic cone $\mathcal C_\infty(Z)$  belongs to $\Hilb$ we must have $h_{\mathcal C_\infty(Z)}=h_{\mathcal C_0(X_\infty)}$, so $\dim R_{\leq\mu}(Z)=\dim R_{[0,\mu]}(\mathcal C_\infty(Z))= \dim R_{[0,\mu]}({p_\infty})$ .  It then follows from an easy induction argument that $\deg(f_\infty^\alpha)=\deg(f^\alpha)$ for all $\alpha$, hence $R_{\leq\mu_*}(Z)$ is spanned by $\{f_\infty^\alpha\}$.

Now we assume $Z$ is a cone limit.  Lemma \ref{lem1.10} and Lemma \ref{l:convexity2} imply by induction that for all $k\in [\underline k, \overline k_j$], $$\log_2\frac{\|\widetilde f^\alpha_{j, {k+1}}\|_{B_{j, 2^{k+1}}}}{\|\widetilde f^\alpha_{j, {k}}\|_{B_{j, 2^k}}}\leq -\deg(f^{\alpha}).$$ But we also have that for given $k$, $$\lim_{j\rightarrow\infty}\log_2\frac{\|\widetilde f^\alpha_{j, {k+1}}\|_{B_{j, 2^{k+1}}}}{\|\widetilde f^\alpha_{j, {k}}\|_{B_{j, 2^k}}}=\lim_{j\rightarrow\infty}\log_2\frac{\|\widetilde f^\alpha_{{k+1}}\|_{B_{j, 2^{k+1}}}}{\|\widetilde f^\alpha_{ {k}}\|_{B_{j, 2^k}}}\geq -\deg(f^{\alpha})-\epsilon_k$$ for $\epsilon_k\rightarrow0$. It follows that if $k_j\rightarrow\infty$, then $$\lim_{j\rightarrow\infty}\log_2\frac{\|\widetilde f^\alpha_{j, {k_j+2}}\|_{B_{j, 2^{k_j+2}}}}{ \|\widetilde f^\alpha_{j, {k_j+1}}\|_{B_{j, 2^{k_j+1}}}}=\lim_{j\rightarrow\infty}\log_2\frac{\|\widetilde f^\alpha_{j, {k_j+1}}\|_{B_{j, 2^{k_j+1}}}}{ \|\widetilde f^\alpha_{j, {k_j}}\|_{B_{j, 2^{k_j}}}}= -\deg(f^{\alpha}).$$ This easily implies that the limit is homogeneous.
\end{proof}

Now as in \cite{DS2} we can do a further perturbation to the above set $\widetilde f^\alpha_{j, k}$. In particular we will replace $\widetilde f^\alpha_{j, {k}}$ by $\widetilde f_{j, {k}}^\alpha+\sum_\beta\epsilon_{\alpha\beta, jk}\widetilde  f_{j, {k}}^\beta$, where $\epsilon_{\alpha\beta, jk}\rightarrow0$ as $j, k\rightarrow\infty, k^{-1}\overline k_j\rightarrow\infty$ and $\epsilon_{\alpha\beta, jk}=0$ unless $\beta<\alpha$ and $\deg(f^\beta)=\deg(f^\alpha)$, such that 
$$\Lambda_{j,k} \widetilde f^\alpha_{j, k}=2^{-\mu_{j,k}}\widetilde f^{\alpha}_{j, {k+1}}+\sum_{\beta<\alpha, \deg(f^{\beta})=\deg(f^\alpha)}\tau_{\alpha\beta, jk}\widetilde f^\beta_{j, {k+1}}, $$
where $\mu_{j,k}\rightarrow \deg(f^\alpha)$ and $\tau_{\alpha\beta, jk}\rightarrow0$ as $j,k\rightarrow\infty$ and $k^{-1}\overline k_j\rightarrow\infty$. The proof is similar to that of Proposition 3.12 and Lemma 3.13 in \cite{DS2}. We omit the details here.

Using the functions $\{\widetilde f^\alpha_{j,k}\}_{\alpha=1}^N$ we obtain holomorphic embeddings $\Phi_{j,k}: B_{j, 2^k}\rightarrow\C^N$. Notice there is a natural inclusion  $B_{j, 2^k}\hookrightarrow B_{j, 2^{k'}}$ whenever $k\geq k'$.  By construction we have $\Phi_{j,k+1}=\Lambda_{j,k}\circ\Phi_{j,k}|_{B_{j, 2^{k+1}}}$, where $\Lambda_{j,k}$ is a linear transformation of $\C^N$ that preserves the Reeb vector field $\xi$ (viewed as a linear vector field on $\C^N$) and satisfies $\Lambda_{j,k}=\Lambda(\Id+\tau_{j,k})$ for a diagonal linear transformation $\Lambda=\text{diag}(2^{\zeta_1}, \cdots, 2^{\zeta_N})$ with $\tau_{j,k}\rightarrow 0$ as $j, k\rightarrow\infty$ and $k^{-1}\overline k_j\rightarrow\infty$.  Furthermore, as $j, k\rightarrow\infty$ and $k^{-1}\overline k_j\rightarrow\infty$, passing to subsequences $\Phi_{j,k}$ converges to  holomorphic embeddings of the cone limits into $\C^N$. 

Let $W$ be the weighted tangent cone in $\C^N$ of $X_\infty$ at $p_\infty$ with respect to $\xi$. From \cite{DS2} and the discussion above we know that $W$ defines a point $[W]$ in  $\Hilb$. So its defining ideal $\mathcal I_W$ is generated by finitely many homogeneous functions $\{G^b\}_{b=1}^{l}$. We list them in degree increasing order, so we have $\deg_\xi(G^b)=d_s$ for $b_{s-1}< b\leq b_{s}$, $b_0=0$ and $d_1\leq d_2\leq...\leq \mu_*$.  It follows that $X_\infty$ is cut out by $\{F^b\}_{b=1}^l$, where $F^b=G^b+H^b$ such that $H^b$ is a local holomorphic function on a neighborhood of $0\in \C^N$ with $\deg_\xi(H^b)>\deg_\xi(G^b)$. Here we  use $\deg_\xi$ to denote the weighted vanishing order at $0$ of a local holomorphic function on $\C^N$.

Now by the discussion at the beginning of this section and the construction of $\Phi_{j,k}$ it is not hard to see that the image of $\Phi_{j, \underline k}$ is cut out by finitely many local holomorphic functions $\{F^b_j\}_{b=1}^l$ such that for each $b$, $F^b_j$ converges locally uniformly to $F^b$. 
 Write $$F_j^b=F^b+S_j^b+H_j^b+L_j^b,$$ where $\deg_\xi(L_j^b)<\deg_\xi(S_j^b)=\deg_\xi(F^b)<\deg_\xi(H_j^b)$ and $S_j^b, H_j^b, L_j^b$ all converge uniformly to $0$ as $j\rightarrow\infty$. Denote $G^b_{j}=G^b+S_j^b$ and  $Q_j^b=G_j^b+L_j^b$.
 For $k\in [\underline k, \overline k_j]$, the image of $\Phi_{j,k}$ is cut out by $\{F^b_{j,k}=0\}$, where $F^b_{j,k}=\Gamma_{j,k}. F^b_j$ for $\Gamma_{j,k}=\Lambda_{j,k}\cdots \Lambda_{j,\underline k+1}$. Denote $G^b_{j,k}=\Gamma_{j,k}.G_j^b$.  
 
Now we try to take limits of the defining functions $\{F_{j, k}^b\}_{b=1}^l$ as $j, k\rightarrow\infty$. We denote by $\mathbb S_{\mu}$ the linear space of all $\xi$-homogeneous polynomials on $\C^N$ with $\xi$-degree equal to $\mu$. We have a natural inner product on $\mathbb S_\mu$ such that an orthonormal basis is given by the set of monomials with coefficient $1$. Denote by $\|\cdot\|_\mu$ the corresponding norm on $\mathbb S_\mu$. By abusing notation we also denote by $\|\cdot\|_\mu$ the induced operator norm on the space of linear transformations on $\mathbb S_\mu$. We also have an induced inner product on the direct sum $\mathbb S_{\leq\mu_*}\equiv \bigoplus_{\mu\leq\mu_*}\mathbb S_\mu$. We denote by $\|\cdot\|_{\bullet}$ the induced norm.  By a linear transformation we may assume $\{G^b_j\}$ is orthonormal in $\mathbb S_{\leq\mu_*}$ with respect to this inner product.
 
Notice $\Lambda$ and $\tau_{j,k}$ naturally act on $\mathbb S_\mu$ for all $\mu$. It is also easy to check that there exits $C>0$ such that for all $\mu$, $\|\Lambda\|_\mu\leq C2^{-\mu}$, and 
\begin{equation}\label{e4-1}\|\tau_{j,k}\|_{\mu}\leq C(1+\|\tau_{j,k}\|_1)^\mu.	
\end{equation}
Since as $j, k\rightarrow\infty, k^{-1}\overline k_j\rightarrow\infty$ we have $\|\tau_{j, k}\|\rightarrow 0$, it is easy to see that for any $b$,  the subsequential limit of $F^b_{j,k}/\|Q^b_{j,k}\|_\bullet$ coincides with the limit of $Q^b_{j,k}/\|Q^b_{j,k}\|_\bullet$, which   is a polynomial of degree at most $\deg_\xi(G^b)$. 
  
 We first consider the case $b\leq b_1$.  It is clear that the subsequential limits of each $F^b_{j,k}/\|Q^b_{j,k}\|_\bullet$ must be homogeneous of degree $d_1$. But the limits for different $b$'s may coincide. As usual we need to perform the Gram-Schmidt transform $\mathcal L_j$ to make an orthonormal basis $\{\widetilde Q^b_{j,k}\}$  out of the linear span of $\{Q^b_{j,k}\}$. Applying the same $\mathcal L_j$ to $\{F^b_{j, k_j}\}$ and $\{G^b_{j, k}\}$ we get $\{\widetilde F^b_{j, k}\}$ and $\{\widetilde G^b_{j, k}\}$.  Notice in fact $G^b_{j, k}=Q^b_{j, k}$ by the minimality of $d_1$. 
 
 \begin{lemma}
 	As $j, k\rightarrow\infty$, subsequential limits of $\{\widetilde F^b_{j,k}\}_{b=1}^{b_1}$ are polynomials on $Z$ with degree $d_1$. If $Z$ is a cone limit, then the limits are homogeneous.  \end{lemma}
 \begin{proof}
It is easy to check using \eqref{e4-1}   that in the Taylor expansion of $\widetilde F^b_{j, k}$, the degree higher than $d_1$ terms will not survive as $j, k\rightarrow\infty$. So $\widetilde F^b_{j, k}$  converges to a polynomial $\widetilde Q^b_\infty$ of degree at most $d_1$, which is the same as the limit of $\widetilde Q^b_{j, k}$. When $Z$ is a cone limit, then  by the minimality of $d_1$ we know $\widetilde Q^b_{j,k}$ must be homogeneous.
 \end{proof}
 
 In particular, since $\mathcal C_\infty(Z) \in \Hilb$, these limit functions  $\{\widetilde Q^b_\infty\}_{b=1}^{b_1}$ precisely span $\mathcal I_{\leq d_1}(Z)$.  

\

 Next we consider the case $b\leq b_2$. We argue as above to extend $\mathcal L_j$ to the linear span of  $\{Q^b_{j, k}\}_{b=1}^{b_2}$.  This time we can only say that the subsequential limits of  $\{\widetilde F^b_{j,k}\}_{b=b_1+1}^{b_2}$ have degree at most $d_2$. Also these limits agree with the limits of $\{\widetilde Q^b_{j,k}\}_{b=b_1+1}^{b_2}$. But from the above discussion  we know that for $b_1<b\leq b_2$, the limit $\widetilde Q^b_\infty$ is orthonormal to $\mathcal I_{\leq d_1}(Z)$ (with respect to the $L^2$ inner product over the unit ball $B(p_Z, 1)$), so  must have degree exactly $d_2$ by dimension counting. Again if $Z$ is a cone limit the limit must be homogeneous; in this case $\widetilde Q^b_\infty$ agrees with the limit $\widetilde G^b_\infty$ of $\widetilde G^b_{j,k}$ and all such limits for $1\leq b\leq b_2$ precisely span $\mathcal I_{\leq d_2}(Z)$.

By a straightforward induction argument, we then obtain linear transformations  $\mathcal L_j$ on the linear span of $\{Q^b_{j, k}\}_{b=1}^l$,  such that 
$\{\mathcal L_j. F_{j, k}^b\}_{b=1}^l$ and  $\{\mathcal L_j. Q_{j, k}^b\}_{b=1}^l$ both converge to an orthonormal basis $\{\widetilde Q_\infty^b\}_{b=1}^l$ of $\mathcal I_{\leq\mu_*}(Z)$ (with respect to the $L^2$ inner product over the unit ball $B(p_Z, 1))$, and $\{\mathcal L_j. G_{j, k}^b\}_{b=1}^l$ converge to a linear independent set $\{\widetilde G_\infty^b\}_{b=1}^l$ of $\mathbb S$.   If $Z$ is a cone limit, then each $\widetilde Q^b_\infty$ is homogeneous and coincides with $\widetilde G^b_\infty$. 

Now we define a sequence of affine schemes $\{\mathcal W_j\}$ in $\C^N$ using the ideal generated by $\{G^b_j\}_{b=1}^l$. Then the above discussion implies that $h_{\mathcal W_j}=h_{\mathcal C_0(X_\infty)}$. In particular $\mathcal W_j$ defines a point $[\mathcal W_j]\in \Hilb$. 
For any  cone limit $\mathcal C\in\mathfrak C'$ which corresponds to a subsequence $k_j\rightarrow\infty$, it then follows that $\Gamma_{j,k_j}. [\mathcal W_j]$ converges to $[\mathcal C]$ as $j\rightarrow\infty$.  

\begin{remark}
When $\mathcal C$ is the tangent cone $\mathcal C_0(X_\infty)$ at $p_\infty$, the arguments above reduce to that of \cite{DS2}. In that case $\mathcal W_j$ is isomorphic to $W$ for all $j$. 	
\end{remark}

As reviewed in Section \ref{ss2-2}, we have the Matsushima theorem which implies that for each $\mathcal C\in\mathfrak C'$, the corresponding element $[\mathcal C]\in \Hilb$ has a reductive stabilizer group in $GL_\xi(N;\C)$. We also have the Bando-Mabuchi uniqueness theorem, therefore Theorem \ref{t:isolation} is a consequence of the following

\begin{proposition}
Let $V$ be a finite dimensional complex representation of a reductive group $G$ and let $\{x_j\}$ be a sequence of points in $\mathbb P(V)$ converging to a limit $x_\infty$. Let $\{I_j\}$ be a sequence of closed intervals in $\mathbb R$ and $g_j: I_j\rightarrow G$ be a sequence of continuous maps. Denote by $\mathcal Y$ the set of all subsequential limits of $g_j(t_j).x_j (g_j\in G, t_j\in I_j)$ as $j\rightarrow\infty$. Suppose that for all $y\in \mathcal Y$, the stabilizer group of $y$ is reductive, then $\mathcal Y$ is contained in a single $G$ orbit.
\end{proposition}
\begin{proof} 
It is not hard to see  we only need to prove that for any given subsequence of $\{j\}$ there is a further subsequence such that all the subsequential limits arising from this subsequence  is contained in a single $G$ orbit. Without loss of generality we may assume the given subsequence is the original sequence.

Clearly $\mathcal Y$ is compact. Choose $y_0\in \mathcal Y$ such that its stabilizer group has minimal dimension among all points in $\mathcal Y$. Without loss of generality we may assume $0\in I_j$, $g_j(0)=\Id$ and $y_0=\lim_{j\rightarrow\infty}x_j$. Denote by $G_0$ the stabilizer group of $y_0$. By assumption $G_0$ is reductive.

Applying the slice theorem at $y$ (see for example \cite{Donaldson2011}) we may find a $G_0$-invariant projective subspace $\mathbb P(V')\subset \mathbb P(V)$ containing $y$, which is transversal  to $G.y$ at $y$ and has the property that for any $y'\in \mathbb P(V')$ close to $y$, locally the intersection of $G.y'$ with $\mathbb P(V')$ coincides with the $G_0$ orbit of $y'$ and furthermore, the identity component of the stabilizer group of $y'$ in $G$ is contained in $G_0$. As a consequence, we can find a neighborhood $\mathcal N$ of $y_0$ in $\mathbb P(V)$ such that the following hold
\begin{itemize}
	\item any $y\in \mathcal Y\cap \mathbb P(V')\cap \mathcal N$ is the limit of points in $G.x_j\cap \mathcal N\cap \mathbb P(V')$;

 \item  for any $z\in U$, $G.z\cap \mathcal N\cap \mathbb P(V')=\bigsqcup_{\alpha=1}^{m_z} G_0.z_\alpha$ for $z_\alpha\in \mathbb P(V')\cap \mathcal N$.
\end{itemize}
 Since the degrees of all the $G$ orbits in $\mathbb P(V)$ are uniformly bounded, one can find $m>0$ such that $m_z\leq m$ for all $z\in \mathcal N$. For $j$ large we then have $G.x_j\cap \mathcal N\cap \mathbb P(V')=\bigsqcup_{\alpha=1}^{m_{x_j}} G_0. x_{j\alpha}$ for $x_{j\alpha}\in \mathbb P(V')\cap \mathcal N$. Passing to a subsequence (noticing this changes $\mathcal Y$) we may assume each $x_{j\alpha}$ converges to a limit $z_\alpha\in \mathbb P(V')$ as $j\rightarrow\infty$. Shrinking $\mathcal N$ we can ensure that there is exactly one $z_\alpha\in \mathbb P(V')\cap \mathcal N$, which is given by $y_0$.

Now we consider the decomposition as $G_0$ representations $V'=V_0\oplus V_1$, where $V_0$ is the invariant subspace of $G_0$. For $j$ large we write $x_{j\alpha}=[v_{j\alpha}, \hat v_{j\alpha}]$, where $v_{j\alpha}\in V_0$ and $\hat v_{j\alpha}\in V_1$. By the choice of $y_0$ we see that any $y=[v, \hat v]\in \mathcal Y\cap \mathbb P(V')\cap \mathcal N$ must satisfy $\hat v=0$. But then it follows that  $y=z_\alpha$ since the action of $G_0$  on $V_0$ is trivial. Hence by the choice of $\mathcal N$ we conclude that $y=y_0$. 

The above argument shows that $y_0$ is an isolated point in $\mathcal Y$. Then  using the fact that $I_j$ is connected  a simple continuity argument then implies that $\mathcal Y$ is contained in a single $G$ orbit, which must be the $G$ orbit of $x_\infty$.  

\end{proof}

\begin{proof}[Proof of Corollary \ref{c:bubble tree finiteness}] From Theorem \ref{t:minimal bubbles} we can define $(Z_1, q_{1})=(X_\infty, p_\infty)$, $(Z_2, q_2)=(\mathcal C_0(X_\infty), O)$. Let $(Z_3, q_{3})$ be the minimal bubble after passing to subsequence, so it is asymptotic to $\mathcal C_0(X_\infty)$. It is then given by the limit corresponding to the subsequence $\{\overline\lambda_j\}$. Denote by $\widetilde Q_\infty^b$ the subsequential limits of $\mathcal L_j. F_{j,\overline\lambda_j}^b$. Then it follows that the ideal of $Z$ in $\mathbb C^N$ is exactly generated by  $\{\widetilde Q^b_\infty\}_{b=1}^l$. Then we can repeat the above discussion in a neighborhood of $q_3$ to obtain the next  minimal bubble. By Theorem \ref{t:volume finite} we know this process stops in finitely many steps. 
  The last bubble limit $(Z_{2k}, q_{2k})$ must then be flat. 
  \end{proof}

\section{Discussion}
\subsection{Algebro-geometric meaning} \label{ss:5-1}
It is a natural question to study the algebro-geometric meaning of the minimal bubbles constructed in this paper. For our purposes we will call these \emph{analytic minimal bubbles}. They are endowed with both algebraic and metric information. 

 Recall that given an affine variety $Y$, in \cite{SZ} we defined the notion of a negative valuation $\nu$ on  $Y$.
\begin{definition}
	A polarized  affine Calabi-Yau  variety  is an affine variety $Y$ with log terminal singularities, which is endowed with a negative valuation $\nu$ whose associated graded ring $Gr_\nu(Y)$ is finitely generated and defines a Fano cone $\mathcal C_\nu$ (in the sense of \cite{CoSz}). The cone $\mathcal C_\nu$ is called the \emph{algebraic asymptotic cone} of $Y$.
\end{definition}

\begin{remark}Notice that a Fano cone itself is by definition a polarized affine Calabi-Yau variety. From the discussion in \cite{SZ} the notion of a polarized affine Calabi-Yau variety is naturally designed to study complete Calabi-Yau metrics with maximal volume growth.
\end{remark}

\begin{definition}
	A   polarized affine Calabi-Yau variety $(Y, \nu)$ is K-semistable (K-polystable) if the associated Fano cone $\mathcal C_\nu$ is K-semistable (K-polystable).
\end{definition}
From the discussion in Section \ref{ss2-3} we know that an analytic minimal bubble $Z$ is always a K-semistable   polarized affine Calabi-Yau variety. Moreover, one expects it to be K-polystable: this is discovered and proved in \cite{SZ} under the quadratic curvature decaying condition, and it is conjectured in \cite{SZ} to be always the case in general. Notice that the weak Calabi-Yau metric on $Z$ is not  uniquely determined by the negative valuation; an explicit example is shown by Chiu \cite{Chiu} for complete Calabi-Yau metrics on $\C^3$ asymptotic to the flat cone $\C^2/\mathbb Z_3\times \mathbb C$. Nevertheless one can still formulate a generalized uniqueness question, see  Conjecture 6.8 in \cite{SZ}, which ensures only a finite dimensional freedom.

  One can ask whether the analytic minimal bubble depends only on the local complex-analytic/algebro-geometric information of the sequence $(X_j, L_j, p_j)$. One can try to  ask a uniqueness question in this general context but there is a complication related to the fact that analytic minimal bubbles are sensitive to the sequence $\{p_j\}$ converging to $p_\infty$. Instead below we will make an attempt formulating it only for algebraic families. 
 This  is meant to be a generalization of the 2-step degeneration theory (see \cite{DS2}) for tangent cones of singularities. Here we explain a strategy and the difficulty involved.  Suppose we are given a flat $\mathbb Q$-Gorenstein family of local germs $\pi: \mathcal X\rightarrow \Delta$ and a holomorphic section $\sigma:\Delta\rightarrow\mathcal X$. Denote $p_t=\sigma(t)$ and $X_t=\pi^{-1}(t)$. We suppose $X_t$ is smooth for $t\neq0$ and $X_0$ has only log terminal singularities. Locally in a neighborhood of $p_0$, we may embed the family holomorphically into the family
$\pi:\C^N\times \mathbb C\rightarrow \C$ such that $p_t$ is mapped to $(0, t)$.  Given positive weights $(\zeta_1, \cdots, \zeta_N)$ on $\mathbb C^N$ and $\zeta_{N+1}=1$ on $\mathbb C$, we can take a weighted tangent cone $\mathcal W$ of $\mathcal X$ at $(0,0)$. The latter is also equipped with a projection map to $\mathbb C$. Denote by $ W_t$ the fiber over $t$. We are only interested in the case when $W_1$ is of pure dimension $n$ and has mild singularities.  Geometrically $W_1$ is given by the limit in $\C^N$ of the rescaled varieties $\lambda(t). X_t$, where $\lambda(t).(z_1, \cdots, z_N)=(|t|^{-\zeta_1}z_1, \cdots, |t|^{-\zeta_N}z_N)$. 

 In general $W_0$ may not agree with the weighted tangent cone $W$ of $X_0$ at $0$. When $\mathcal X$ comes from  degeneration of polarized K\"ahler-Einstein varieties, one can choose the embedding above and the weights $(\zeta_j)$ such that $W$ is the intermediate K-semistable Fano cone discovered in \cite{DS2}, and there is a further equivariant degeneration to the metric tangent cone of $X_0$ at $p_0$. More generally, 
by Xu-Zhuang \cite{XZ} we can always make $W$ a K-semistable Fano cone in an essentially unique way. Denote such a $W$ by $W^{ss}$. By Li-Wang-Xu \cite{LWX} there is an equivariant degeneration from $W^{ss}$ to a unique K-polystable Fano cone $\mathcal C^{ps}$.

Now we choose $\lambda>0$ and adjust the weights $(\zeta_j)$ to $(\lambda \zeta_j)$. It seems possible to show that a suitable chosen $\lambda$ would give rise to a limit $W_1$ whose weighted asymptotic cone at infinity (with respect to the inverse weights $(-\zeta_j)$) is given by $W^{ss}$. Moreover, one expects such a $W_1$ to be essentially unique. Suppose all these work out we call $W_1$ an \emph{algebraic minimal bubble}. It is a K-semistable   polarized affine Calabi-Yau variety.

It is tempting to relate the algebraic minimal bubble with the analytic minimal bubble. However, the situation is more subtle than one would naively think. First as mentioned above an analytic minimal bubble $Z$ is expected to be K-\emph{polystable} (rather than only K-\emph{semistable}). This means that in general we need a further degeneration from $W_1$ to $Z$. To speculate what this should be one can go back to the proof in Section 4, where we have the error terms $\tau_{j,k}$ which goes to zero as $j, k, k^{-1}\overline k_j\rightarrow\infty$. These may accumulate and diverge when we compose from $\underline k$ to $\overline k_j$. So naturally one expects to embed both $W_1$ and $Z$ into $\C^N$ and a sequence of elements $\Phi_j\in GL_\xi(N)$ (where $\xi$ is the diagonal linear vector field associated to the weights $(\zeta_j)$) such that $\Phi_j. W_1$ converges to $Z$ in an appropriate sense. The subtlety here, according to the author, is that if one only takes purely algebro-geometric consideration this way, then such a limit $Z$ may not be unique in general. This is related to the discussion in \cite{SZ} that  the class of all K-polystable   polarized affine Calabi-Yau varieties with a fixed algebraic asymptotic cone modulo isomorphisms may not be Hausdorff in general. Also to determine the metric on the analytic minimal bubble from the algebraic information is another issue, since as mentioned above, the Calabi-Yau metric may not be uniquely determined by the underlying   polarized affine Calabi-Yau variety. Nevertheless, we make a tentative conjecture here, which may be subject to modification with further progress. 

\begin{conjecture}[Algebraicity of minimal bubbles]
	Given a family $\mathcal X$ as above, then 
	\begin{enumerate}
		\item There exists a unique algebraic minimal bubble $W_1$;
		\item There exists a unique K-polystable   polarized affine Calabi-Yau variety $Z$ such that the following  properties hold
		\begin{enumerate}
			\item one can find affine embeddings of $W_1$ and $Z$ into a common $\C^N$ such that both negative valuations are induced by a weight vector $(\zeta_1, \cdots, \zeta_N)$ and a one parameter subgroup $\lambda(u)\in GL_\xi(N)$ such that $Z=\lim_{u\rightarrow 0}\lambda(u). W_1$;
			\item The algebraic asymptotic cone of $Z$ is given by $\mathcal C^{ps}$;		\item Suppose the family comes from a polarized K\"ahler-Einstein degeneration, then the analytic minimal bubble is isomorphic to $Z$ as   polarized affine Calabi-Yau varieties.
		\end{enumerate}
		\item There exists a unique complete (weak) Calabi-Yau metric on $Z$ up to holomorphic isometry and metric scaling, which gives the metric on the analytic minimal bubble associated to a polarized K\"ahler-Eintein degeneration. 
	\end{enumerate} 
\end{conjecture}
\begin{remark}
A version of this conjecture was formulated in a discussion with Cristiano Spotti in 2016. For low dimensional examples, progress has been made by him and Martin de Borbon \cite{SM}.
\end{remark}

\begin{remark}
From the above discussion we know that in item (2), the uniqueness does not hold if we only assume (a) and (b) but it seems to be only a finite dimensional freedom. Also it is possible to use the more intrinsic language of valuations instead of affine embeddings in $\C^N$ but here we will not delve into it.  
The main purpose of this conjecture is to motivate further study.
\end{remark}

\begin{remark}
One may consider even more general $\mathbb Q$-Gorenstein families of log terminal singularities where the general fibers are not necessarily smooth. In this case one needs to assume the volume density at $p_0$ (in the sense of \cite{CLi}) is strictly smaller than that of $p_t$ so that bubbling actually occurs.	It is straightforward to extend the above discussion in this generality, which we omit here. 
\end{remark}

\begin{remark}
One may also ask about the dependence of both $W_1$ and $Z$ on the section $\sigma$ and the point $p_0$. For example, is it generically the same?
\end{remark}

To show this is of interest, we consider an  example here. Let $\mathcal X$ be the family defined by $x_1^2+x_2^2+x_3^2+x_4^4+tx_4=0$ in $\C^4\times \C$ and $p_t=(0, 0, 0, 0, t)$. Notice the central fiber $X_0$ is itself a K-semistable Fano cone with weights $\zeta_1=\zeta_2=\zeta_3=2, \zeta_4=1$ (see for example the discussion in \cite{DS2} on $A_k$ singularities). The corresponding K-polystable Fano cone $\mathcal C^{ps}$ is given by $\mathbb C^2/\mathbb Z_2\times \mathbb C$, which admits  the flat orbifold metric.  The obvious algebraic minimal bubble is given by $W_1=X_1$ if we choose $\zeta_1=\zeta_2=\zeta_3=\frac{2}{3}, \zeta_4=\frac{1}{3}$. Suppose $\mathcal X$ arises from  degeneration of polarized K\"ahler-Einstein metrics (for example, of polarized Calabi-Yau varieties), then the main results in this paper seem to  imply one of the following two possibilities must occur 
\begin{enumerate}
	\item The analytic minimal bubble is $W_1$ itself equipped with a complete Calabi-Yau metric which is asymptotic to the above flat orbifold cone. But we do not expect this to be the case since it would violate the general no-semistability at infinity conjecture in \cite{SZ}. 
	\item There is a degeneration of $W_1$ to another   polarized affine Calabi-Yau variety $Z$ with algebraic asymptotic cone $\mathcal C^{ps}$ and the analytic minimal bubble is given by $Z$. Algebraically, one way is to use the one-parameter subgroup $\lambda(u)=(u^2, u^2, u^2, u^4)$ and $Z$ is given by $x_1^2+x_2^2+x_3^2+x_4=0$, which is simply $\mathbb C^3$ with a non-standard negative valuation.  We know the latter admits an essentially unique complete Ricci-flat metric asymptotic to  the flat orbifold cone, by \cite{YLi, CR,Gabor1,Gabor2}.
\end{enumerate}
We leave these for interested readers to explore further.

\subsection{More general situation}

It is desirable to  relax the assumptions used in this paper. Our arguments depend crucially on the existence of a polarization, which allows the application of the H\"ormander $L^2$ technique. Without this the following is an important open question. Let $(X_j, \omega_j, p_j)$ be a sequence of $n$ dimensional K\"ahler-Einstein manifolds with $Ric(\omega_j)=c_j\omega_j (|c_j|\leq 1)$ such that $\Vol(B(p_j, 1))\geq \kappa>0$ and suppose $(X_\infty, d_\infty, p_\infty)$ is a Gromov-Hausdorff limit. Denote by $X_\infty=X_\infty^{reg}\cup X_\infty^{sing}$ the regular-singular decomposition in the sense of Cheeger-Colding. Notice that $X_\infty^{reg}$ is an open K\"ahler manifold. Denote by $\mathcal O$ the push-forward to  $X_\infty$ of the structure sheaf of $X_\infty^{reg}$. 
\begin{conjecture}[Complex analyticity]
	$(X_\infty, \mathcal O)$ is a normal complex-analytic space and the singular set in the complex-analytic sense coincides with $X_\infty^{sing}$.
\end{conjecture}
Noticing that by Liu-Szekelyhidi \cite{LZ} cone limits in this general case are still  affine varieties, and it is natural to conjecture that the singularities are log terminal. The latter seems to be a necessary ingredient in order to extend Theorem \ref{t:volume finite}.

Moving further, one can also study the bubbling phenomenon in this more general situation.
 Notice that here we do not expect the bubble limits to be affine varieties since compact holomorphic cycles may get contracted in the limiting process. A model case is to consider the convergence to the asymptotic cone of a complete Calabi-Yau metric with maximal volume growth. It is also an open question to classify the later in terms of algebro-geometric data, see \cite{SZ} for more details.

Finally, the volume collapsing situation is substantially much more complicated.  The only known general  result is in the case of 2 dimensional Calabi-Yau metrics with uniformly bounded $L^2$ curvature energy studied in \cite{SZ2}. In particular, in this more restrictive case (which corresponds to degenerations of K3 surfaces) we know the bubble tree is also finite, but the algebro-geometric meaning is not yet well-understood. In general one expects that non-Archimedean geometry may be involved to study the bubble tree structure. Without the bounded energy hypothesis, even in complex dimension 2, if one allows collapsing then in the local setting the bubble tree may not even terminate in finite steps.

 \subsection{The Riemannian setting}\label{ss:5-2}
 The notion of  a \emph{bubble tree} is commonly studied in many problems in geometric analysis, such as minimal surfaces, harmonic maps, Yang-Mills connections and Einstein metrics etc. Often one works with an energy functional which is conformally invariant and studies the singularity formation for a sequence of solutions with uniformly bounded energy. Specific to the case of Einstein metrics in real $m$ dimensions, assuming a uniform bound on the curvature energy $\int|Rm|^{m/2}$ and non-collapsing volume, bubble tree convergence was studied in \cite{AC, Bando1, Bando2}; in dimension 4, by the result of Cheeger-Naber \cite{CN} one gets the same results under the volume non-collapsing assumption alone. In higher dimensions, very little is known in general regarding the bubble tree structure if we only assume non-collapsing volume. We formulate some conjectures and questions related to generalizations of our main results in this paper.

First we will allow Einstein metrics with mild singularities. For simplicity by \emph{singular} Einstein metrics we mean a non-collapsed  Gromov-Hausdorff limit of smooth Einstein manifolds.
 Given $m\geq3$, we denote by $\mathfrak {C}(m)$ the isometry classes of $m$ dimensional singular Ricci-flat cones, and by $\mathcal E(m)$ the isometry classes of $m-1$ dimensional compact singular Einstein metrics with Einstein constant $m-2$. For $X\in \mathcal E(m)$ we denote by $\Vol(X)$ the Riemannian volume of $X$ and for $\mathcal C\in \mathfrak C(m)$ we denote by $\Vol(\mathcal C)$ the volume density of $\mathcal C$. A natural generalization of Theorem \ref{t:volume finite} is 
 
 \begin{conjecture}[Volume rigidity]\label{c:volume discrete general}
 For any $m\geq3$, the set 
 $$\mathcal V_1(m)=\{\Vol(\mathcal C)|\mathcal C\in \mathfrak{C}(m)\}$$ and $$\mathcal V_2(m)=\{\Vol(X)|X\in \mathcal E(m))\}$$  are both discrete subsets of $(0, \infty)$.
 \end{conjecture}
 \begin{remark}
One might also want to allow even more general singularities using more synthetic formulations, without referring to Gromov-Hausdorff limits. However for the conjecture to hold one has to exclude for example codimension 2 cone singularities; otherwise one notices that the volume of flat 2 dimensional cones can deform continuously. \end{remark}

This is known to be true if we restrict to cones with isolated singularities and smooth $X$. This follows from the local analytic structure of the moduli space of smooth Einstein metrics (see for example Besse \cite{Besse}, Chapter 12). A strategy to tackle this question in general would require more precise understanding of the singularity formations in one dimension lower, which suggests a possible (but possibly complicated) inductive approach. 
 
Now we consider a sequence of Einstein metrics $(X_j, g_j, p_j)$ converging to a singular Einstein metric $(X_\infty, g_\infty, p_\infty)$ with non-collapsing volume. Denote by $\Vol(p_\infty)$ the volume density at $p_\infty$. A folklore question is the uniqueness of tangent cones at $p_\infty$. This is  known if one a priori assumes the curvature has only quadratic blow-up at $p_\infty$ (c.f. \cite{CM}).  As a generalization of Theorem \ref{t:minimal bubbles}, one needs to study

 \begin{conjecture}[Cone rigidity]\label{c:cone rigidity}
The set of bubble limits with volume density equal to $\Vol(p_\infty)$ consists of a single point.  \end{conjecture}

This can be viewed as an upgrade of the uniqueness of tangent cones. It seems natural to apply the strategy in \cite{CM} to prove this  when the tangent cone at $p_\infty$ is smooth. Notice Conjecture \ref{c:volume discrete general} and \ref{c:cone rigidity} together imply the existence of minimal bubbles and finite step termination of bubble trees, in exactly the same way as the K\"ahler-Einstein case.

Finally, a key reason why one can prove far stronger results in the (polarized) K\"ahler setting is due to the rigidity of holomorphic spectrum (as discussed in \cite{DS2}). In the general Riemannian setting it is desirable to find a substitute. A naive candidate is the Laplacian spectrum. We propose the following questions for future study

\begin{question}
Are there special arithmetic properties of the Laplacian spectrum on compact Einstein manifolds with Einstein constant $m-1$?	
\end{question}

\begin{question}[Spectrum rigidity]
Suppose $g_t$ is a smooth family of Einstein metrics on a compact manifold with Einstein constant $m-1$, is it true that the Laplacian spectrum must be constant in $t$?
\end{question}

Of course, one can formulate analogous conjectures/questions to the above for other geometric PDEs.

\bibliographystyle{plain}

\bibliography{ref.bib}

\end{document}